\documentclass[aos,preprint,reqno]{imsart} 
\setattribute{journal}{name}{}

\usepackage[round, sort&compress]{natbib}
\usepackage[breaklinks]{hyperref}
\usepackage{url}
\usepackage{amsmath,amsfonts,amssymb,amsthm}
\usepackage{subcaption}
\usepackage{fullpage}
\usepackage{graphicx}
\usepackage{verbatim}
\usepackage{booktabs}
\usepackage[noend]{algpseudocode}
\usepackage{algorithm, mdframed}
\usepackage[normalem]{ulem}
\usepackage{stmaryrd}

\floatstyle{plain}
\newfloat{myalgo}{tbhp}{mya}

\newenvironment{Algorithm}[1][tbh]%
{\begin{myalgo}[#1]
\centering
\begin{minipage}{.65 \textwidth}
\begin{algorithm}[H]}%
{\end{algorithm}
\end{minipage}
\end{myalgo}}

\definecolor{darkred}{rgb}{0.5,0,0}
\definecolor{darkgreen}{rgb}{0, 0.3,0}
\definecolor{darkblue}{rgb}{0,0,0.6}
\definecolor{LightGray}{rgb}{.6,.6,.6}

\DeclareMathOperator{\Uniform}{Uniform}
\DeclareMathOperator{\Bernoulli}{Bernoulli}
\DeclareMathOperator{\Binomial}{Binomial}
\DeclareMathOperator{\MSE}{MSE}
\DeclareMathOperator{\MISE}{MISE}
\newcommand{\Bin}{\mathcal{B}}
\newcommand{\given}{\,|\,}

\newcommand{\fourEvent}[4]{{\mathfrak{F}_{#1, #2}(#3, #4)}}

\newcommand{\PP}{\ensuremath{\mathcal{P}}}
\newcommand{\QQ}{\ensuremath{\mathcal{Q}}}
\newcommand{\SSS}{\ensuremath{\mathcal{S}}}

\newcommand{\defn}[1]{\textbf{#1}}
\newcommand{\defas}{:=}
\newcommand{\st}{{\,:\,}}
\newcommand{\Nats}{\Naturals}
\newcommand{\eqdist}{\stackrel{\mathrm{d}}{=}}
\newcommand{\iid}{\stackrel{\text{iid}}{\sim}}
\newcommand{\ind}{\stackrel{\text{ind}}{\sim}}

\def\EM#1{\ensuremath{#1}}
\def\mbb#1{\EM{\mathbb{#1}}}

\def\Naturals{{\EM{{\mbb{N}}}}}

\newcommand{\Gd}{{G^\dagger}}

\newcommand{\GG}{\mathfrak{G}}

\newcommand{\G}{\mbb{G}}

\newcommand{\Expect}{\mbb{E}}

\newcommand{\MSEval}{\MSE}
\newcommand{\MSEfun}{\MISE}

\newcommand{\ER}{Erd\H{o}s--R\'{e}nyi}

\theoremstyle{plain}
\newtheorem{proposition}{Proposition}[section]
\newtheorem{lemma}[proposition]{Lemma}
\newtheorem{corollary}[proposition]{Corollary}

\newtheorem{theorem}[proposition]{Theorem}

\theoremstyle{definition}
\newtheorem{definition}[proposition]{Definition}

\theoremstyle{remark}

\renewcommand{\hat}{\widehat} 

\def\set#1{\{ #1 \}}
\def\norm#1{\mathopen\| #1 \mathclose\|}		

\newcommand{\zo}{\EM{\set{0,1}}}
\newcommand{\zov}{\zo-valued}

\begin{document}

\begin{frontmatter}
\title{An Iterative Step-Function Estimator for Graphons}

\begin{aug}
\author{\fnms{Diana Cai${}^{\dagger,\ddagger}$}\ead[label=e1]{dcai@post.harvard.edu}},
\author{\fnms{Nathanael Ackerman${}^\ddagger$}\ead[label=e2]{nate@math.harvard.edu}},
\and
\author{\fnms{Cameron Freer${}^{\dagger,\star}$}\ead[label=e3]{freer@mit.edu}}
\affiliation{\thanksref[$\dagger$]{gamalon}\!Gamalon,
	\thanksref[$\ddagger$]{harvard}\!Harvard University, and
	\thanksref[$\star$]{mit}\!Massachusetts Institute of Technology}
\end{aug}

\begin{abstract}
Exchangeable graphs arise via a sampling procedure
from measurable functions known as \emph{graphons}.
A natural estimation problem is how well we can recover a graphon
given a single graph sampled from it.
One general framework for estimating a graphon uses
step-functions obtained by partitioning the nodes of the graph according
to some clustering algorithm.
We propose an \emph{iterative step-function estimator} (ISFE) that, given an
initial partition, iteratively clusters nodes based on their edge densities
with respect to the previous iteration's partition.
We analyze ISFE and demonstrate its performance
in comparison with other graphon estimation techniques.
\end{abstract}


\end{frontmatter}


\section{Introduction}
\label{intro}

Latent variable models of graphs can be used to model hidden
structure in large networks and have been applied to a variety of problems such as
community detection \citep{MR1908073}
and link prediction \citep{DBLP:conf/nips/MillerGJ09}.
Furthermore, many graphs are naturally modeled as exchangeable
when the nodes have no particular ordering \citep{DBLP:dblp_conf/nips/Hoff07}.
Examples of exchangeable graph models include the stochastic block model (SBM)
\citep{MR718088} and its extensions \citep{C.Kemp:2006:53fd9},
latent feature models
\citep{DBLP:conf/nips/MillerGJ09,DBLP:conf/icml/PallaKG12},
and latent distance models \citep{MR1951262}.

Several key inference problems in exchangeable graph models
can be formulated in terms of estimating
symmetric measurable functions $W\colon [0,1]^2\rightarrow[0,1]$, known as
\emph{graphons}.
There is a natural sampling procedure that produces an exchangeable (undirected)
random graph from a graphon $W$ by first sampling
a countably infinite set of independent uniform random variables
$\{U_i\}_{i\in\Nats}$, and then sampling an edge between
every pair of distinct vertices $i$ and $j$ according to an independent
Bernoulli random variable with weight
$W(U_i, U_j)$.
In the case where the graphon is constant or piecewise constant with a
finite number of pieces, this procedure recovers the standard notions of
\ER\ graphs and stochastic block models, respectively. But this procedure is
much more general; indeed,
\citet{MR637937} and \citet{Hoover79} showed, via what can be viewed as a
higher-dimensional analogue of de~Finetti's theorem,
that the distribution of any exchangeable graph is some mixture of such
sampling procedures from graphons.

Graphon estimation has been studied in two contexts:
(1) \emph{graphon function estimation}
\citep{2013arXiv1309.5936W, MR3161460},
where we are concerned with
inverting the entire sampling procedure to recover a measurable
function from a single sampled graph,
and
(2) \emph{graphon value estimation},
where we are interested in inverting just the second step of the sampling procedure,
to obtain estimates of the latent values $W(U_i, U_j)$ from a single graph
\citep{2012arXiv1212.1247C,
2014arXiv1410.5837G}
(or several \citep{DBLP:conf/nips/AiroldiCC13}) sampled using the sequence
$\{U_i\}_{i\in\Nats}$.

Graphons are well-approximated by step-functions in the \emph{cut distance}
\citep{MR1674741, MR1723039, MR3012035}, a notion of distance between graphs
that extends to graphons,
which we describe in Section \ref{background}.
Although the topology on the collection of graphons induced by the cut distance is coarser than
that induced by $L^2$ (as used in MSE and MISE risk), two graphons are close in
the cut distance precisely when their random samples (after reordering) differ
by a small fraction of edges.
Hence it is natural to consider graphon estimators that produce step-functions;
this has been extensively studied with the stochastic block model.

A standard approach to approximating graphons using step-functions is to first
partition the vertices of the sampled graph $G$ and then return the
step-function graphon determined by the average edge
densities in $G$ between classes of the partition.
In this way, every clustering algorithm can be seen to induce a graphon
estimation procedure (Section~\ref{clustering}).
While many clustering algorithms thereby give rise to
tractable graphon estimators, one challenge is to produce clustering algorithms that
induce good estimators. In this paper, we introduce a method, motivated by the
cut distance, that takes a vertex partition and produces another partition that
yields an improved graphon estimate. By iterating this method, even better
estimates can be obtained. We describe and analyze the graphon estimator that
results from this iterative procedure applied to the result of a clustering
algorithm.

\subsection{Contributions}
We propose
iterative step-function estimation (ISFE),
a computationally tractable graphon estimation procedure motivated by
the goal, suggested by the Frieze--Kannan weak regularity lemma, of
finding a partition that induces a step-function estimate close in cut distance
to the original graphon.
ISFE iteratively improves a partition of the vertices of the sampled
graph by considering the
average edge densities between each vertex and each of the classes of the existing
partition
(Section \ref{isfe}).

We analyze a variant of ISFE on graphs sampled from a $2$-step stochastic block model,
and demonstrate a sense in which ISFE correctly classifies an arbitrarily large fraction of the vertices, as the number of vertices of the sampled graph and number of classes in the partition increase (Section~\ref{isfe-analysis}).

Finally, we evaluate our graphon estimation method on data sampled from several
graphons,
comparing ISFE against several other graphon estimation methods
(Section~\ref{evaluation}).
ISFE quickly recovers detailed structure in samples from graphons
having block structure, while still performing
competitively with other tractable graphon estimators on various classes of
continuous graphons, while making fewer structural assumptions.

\section{Background and related work}
\label{background}

Throughout this paper, graphs are undirected and simple; we consider
sequences of graphs that are \emph{dense}, in that a graph with $n$ vertices
has $\Omega(n^2)$ edges.
For natural numbers $n\ge 1$, we define a \emph{graph on $[n]$} to be a
graph with set of vertices
$[n] \defas \{1, \ldots, n\}$;
its adjacency matrix is the $\{0,1\}$-valued $n\times n$ matrix
$(G_{ij})_{i,j\in[n]}$, where $G_{ij} = 1$ iff $G$ has an edge between
vertices $i$ and $j$. Graphs on $\Nats$, and their adjacency matrices, are
defined similarly.
We write $x \eqdist y$ when two random variables $x$ and $y$ are equal in
distribution, and abbreviate \emph{almost surely} and \emph{almost everywhere} by
a.s.\ and a.e., respectively.

\subsection{Graphons} For detailed background on graphons and the relationship
between graphons and exchangeable graphs, see
the book by \citet{MR3012035} and surveys by \citet{MR2463439} and \citet{MR2426176}.
Here we briefly present the key facts that we will use.

A random graph $G$ on $\Nats$ is \defn{exchangeable} when its distribution is
invariant under arbitrary permutations of $\Nats$. In particular, if such a
graph is not a.s.\ empty, then the marginal probability of an edge between any
two vertices is positive.

\begin{definition}
A random \zov\ array $(A_{ij})_{i,j\in \Nats}$ is (jointly) \defn{exchangeable} when
\begin{align}
    (A_{ij}) \eqdist (A_{\sigma(i), \sigma(j)})
\end{align}
for every permutation $\sigma$ of $\Nats$.
\end{definition}
Note that a random graph on $\Nats$ is exchangeable precisely when its adjacency
matrix is jointly exchangeable.
We now define a sampling procedure that produces exchangeable graphs.

\begin{definition}
A \defn{graphon} $W$ is a symmetric measurable function $[0,1]^2 \to [0,1]$.
\end{definition}

A graphon can be thought of as a continuum-sized, edge-weighted graph.
We can sample from a graphon in the following way.

\begin{definition}
Let $W$ be a graphon.
The \defn{$W$-random graph} on $\Nats$, written $\G(\Nats, W)$, has adjacency
matrix $(G_{ij})_{i, j\in\Nats}$ given by the following sampling procedure:
\begin{align}
\begin{split}
    U_i    & \iid \Uniform[0, 1] \\
    G_{ij} \ | \ U_i, U_j & \ind \Bernoulli(W(U_i, U_j))  \textrm{~for~} i<j
\end{split}
\label{graphon-sample}
\end{align}
For $n\ge 1$, the random graph $\G(n, W)$ on $[n]$ is formed similarly.
\end{definition}
Every $W$-random graph is exchangeable, as is any mixture of $W$-random graphs.
Conversely, the following statement is implied by the Aldous--Hoover theorem, a
two-dimensional generalization of
de~Finetti's theorem, which characterizes exchangeable sequences as mixtures of i.i.d.\ sequences.

\begin{theorem}[{\citet{MR637937,Hoover79}}]
\label{aldous-hoover}
Suppose $G$ is an exchangeable graph on $\Nats$. Then $G$ can be written as the
mixture of $W$-random graphs $\G(\Nats, W)$  for some probability measure on graphons $W$.
\end{theorem}
The Aldous--Hoover representation has since been extended to higher dimensions,
more general spaces of random
variables, and weaker notions of symmetry; for a detailed presentation,
see \citet{MR2161313}.

Since every exchangeable graph is a mixture of graphon sampling procedures,
many network models can be described in this way
\citep{DBLP:dblp_conf/nips/Hoff07}.
The stochastic block model \citep{MR718088} is such an example, as explored further by
\citet{Bickel15122009} and others;
it plays a special role as one of the simplest models that can approximate
arbitrary graphon sampling procedures.
Some Bayesian nonparametric models, including the
eigenmodel \citep{DBLP:dblp_conf/nips/Hoff07},
Mondrian process graph model \citep{NIPS2008_3622}, and
random function model \citep{DBLP:conf/nips/LloydOGR12} were built knowing
the Aldous--Hoover representation.
Furthermore, many other such models are
naturally expressed in terms of a distribution on graphons
\citep{DBLP:journals/pami/OrbR14, DBLP:conf/nips/LloydOGR12},
including
the infinite relational model (IRM) \citep{C.Kemp:2006:53fd9}
the latent feature relational model (LFRM) \citep{DBLP:conf/nips/MillerGJ09}, and
the infinite latent attribute model (ILA) \citep{DBLP:conf/icml/PallaKG12}.

Two different graphons $W_0, W_1$ can give rise to the same distribution on graphs,
in which case we say that $W_0$ and $W_1$ are \defn{weakly isomorphic}.
For example, modifying a graphon on a measure zero subset does not change the
distribution on graphs. Moreover, applying a measure-preserving transformation
to the unit interval, before sampling the graphon,
leaves the distribution on graphs unchanged. The following is a
consequence of Proposition 7.10 and Equation (10.3) of \citet{MR3012035}.

\begin{proposition}
Let $\varphi: [0,1] \rightarrow [0,1]$ be a measure-preserving transformation,
i.e., a map such that
$\varphi(U)$ is
uniformly distributed
for $U\sim \Uniform[0,1]$.
Then the graphon $W^\varphi$ defined by
    $W^{\varphi}(x,y) = W(\varphi(x), \varphi(y))$
is weakly isomorphic to $W$.
\end{proposition}

Thus, the graphon from which an exchangeable graph is sampled is
non-identifiable; see \citet[{\S III.D}]{DBLP:journals/pami/OrbR14}.
Such measure-preserving transformations are essentially the only freedom allowed.
Hence the appropriate object to estimate is a graphon
\emph{up to weak isomorphism}.
\begin{theorem}[\cite{Hoover79}]
\label{hoover}
If $W_0$ and $W_1$ are weakly isomorphic, then there are measure-preserving
transformations $\varphi_0, \varphi_1$ and a graphon $V$ such that
$W_0^{\varphi_0} = W_1^{\varphi_1} = V$ a.e.
\end{theorem}
As a result of Theorem \ref{hoover}, when considering the problem of estimating
a graphon, we only ask to recover the graphon up to a measure-preserving
transformation; this is analogous to a key aspect of the definitions of
\emph{cut distance} and of \emph{$L^2$ distance} between graphons, which we
describe in Appendix~\ref{appendix-risk}.

\subsection{The graphon estimation problems}

Given a graph with adjacency matrix $(G_{ij})$ sampled according
to Equation~\eqref{graphon-sample}, there are two natural ways one may
seek to invert this sampling procedure.
Here we consider two distinct graphon estimation problems that correspond to
inverting one or both of the sampling steps.
The ``graphon value estimation problem'' aims to invert the second step
of the sampling procedure, and hence can be thought of as finding the local
underlying structure of a graph sampled from a graphon (without concluding
anything about the graphon at any location not involved in the sample).
Suppose we sample the $W$-random graph $\G(n, W)$ using
$\{U_i\}_{i\in [n]}$ as in Equation~\eqref{graphon-sample}.
Graphon value estimation consists of giving an estimator
$\hat{M} \defas (\hat{M}_{ij})_{i,j \in [n]}$ for the matrix
$M \defas (M_{ij})_{i,j\in [n]}$ where each $M_{ij} \defas W(U_i, U_j)$.
One measure of success for the graphon value estimation problem
is given by the mean squared error:
\begin{align}
\label{mse-eq}
\MSEval(\hat{M}) \defas
\Expect \
\Bigl(\frac1{n^2} \sum_{i=1}^n \sum_{j=1}^n
\bigl (M_{ij} -\hat{M}_{ij} \bigr)^2 \Bigr),
\end{align}
as used by \citet{2012arXiv1212.1247C}
and \citet{2014arXiv1410.5837G} (see also
\citet{DBLP:conf/nips/AiroldiCC13}).
Whereas MSE in nonparametric function estimation is typically with respect to
particular points of the domain (see, e.g., \citet[\S1.2.1]{MR2724359}), here
the random sequence $\{U_i\}_{i\in[n]}$ is latent, and so
we take the expectation also with respect to the randomness in the terms $U_i$
(and hence in the terms $M_{ij}$), following
\citet[\S2.6]{2012arXiv1212.1247C}.

The ``graphon function estimation problem'' aims to invert the entire sampling
procedure to recover a graphon (i.e., symmetric measurable function).
A notion of success for the graphon function estimator problem, used by
\citet{2013arXiv1309.5936W}, \citet{MR3161460}, and \citet{Olhede01102014},
is given by the mean integrated squared error for an estimator $\hat{W}$ of a graphon $W$:
\begin{align*}
\MSEfun(\hat{W}) &\defas
\Expect \
\inf_\varphi \ \norm{W - \hat{W}^\varphi}_2
\\
&=
\Expect \
\inf_\varphi \int_{[0,1]^2}
\bigl(W(x,y) - \hat{W}^\varphi(x,y) \bigr) ^2\, dx\, dy,
\end{align*}
where $\varphi$ ranges over measure-preserving transformations of $[0,1]$.
However, as we describe in
Appendix~\ref{appendix-l2-cut},
there are graphons $W$ and $V$ such that the random graphs $\G(\Nats, W)$ and
$\G(\Nats, V)$ are close in distribution, but $W$ and $V$ are far in
$L^2$ distance.
An alternative global notion of success for the function estimation problem is
to use the distribution of such random graphs directly
\citep{MR1702867}, or to use the cut distance, defined in terms of the cut along
which two graphs differ the most in their edge densities,
which also captures this notion of subsamples being close in distribution;
see
Appendix~\ref{appendix-l2-cut}.

The distinction between these two problems is analogous to the typical distinction between MSE and MISE in nonparametric function estimation \citep{MR2724359};
see also the two estimation problems in \citet[\S2.1]{DBLP:conf/aistats/YangHA14}.

In general, it is impossible to recover a measurable function from its values
at a countable number of points. However, if we assume that the measurable
function has specific structure (e.g., is continuous, Lipschitz, a step-function, etc.), then it may become
possible.
As a result, many graphon estimation methods, which we describe below,
require the graphon to have a representation of a certain form.
However, the problem of recovering a real-valued function from its values at a
random set of inputs, under various assumptions on the function, may be treated
separately from the estimation of these values.
Hence in this paper, while we illustrate the step-function graphon provided by
ISFE, we evaluate its graphon value estimate using MSE.

\subsection{Graphon estimation methods}

The first study of graphon estimation was by \citet{MR1702867} in the more
general context of exchangeable arrays.  This work predates the development of
the theory of graphons; for details, see
\citet[\S V]{DBLP:journals/pami/OrbR14}.

A number of graphon estimators have been proposed in recent years.
Here we mention several that are most closely related to our approach.
The stochastic block model approximation (SBA) \citep{DBLP:conf/nips/AiroldiCC13}
requires multiple samples on the same vertex set, but is similar to our
approach in some respects, as it partitions the vertex set according
to the $L^2$ metric on their edge vectors (in essence, the vector of average
edge densities with respect to the discrete partition).
Sorting and smoothing (SAS)
\citep{DBLP:conf/icml/ChanA14} takes a different approach to providing a
computational tractable estimator, requiring the graphon to have absolutely
continuous degree distribution.

Several estimators use spectral methods, including universal singular value
thresholding (USVT) \citep{2012arXiv1212.1247C}.
Rather than estimating a specific cluster and using this to define a step-function,
\citet{2014arXiv1406.5647A}
first estimate a co-cluster matrix and then obtain a graphon estimate from this
matrix by using eigenvalue truncation and $k$-means.

Other recent work in graphon estimation has focused on minimax optimality,
histogram bin width, estimation using moments, or consequences of the graphon
satisfying certain Lipschitz or H\"older conditions
\citep{DBLP:conf/aistats/YangHA14,
2013arXiv1309.5936W,
Olhede01102014,
Bickel15122009,
MR2906868,
MR3161460}.

The estimation problem for latent space models can also be seen as graphon
estimation, as such models are equivalent to graphon sampling procedures for graphons having nicer properties
than mere measurability
\citep[\S2.4]{2012arXiv1212.1247C}.

Many of the above graphon estimators are formulated in the setting of bipartite
graphs and \emph{separate} exchangeability, where the distribution
is invariant under separate permutations of the rows and columns.
For notational simplicity, we focus on the case of arbitrary undirected graphs,
whose adjacency matrices are symmetric, and for which joint exchangeability is
the appropriate notion, but many of our results have straightforward analogues
for bipartite graphs.

\section{Iterative step-function estimation}
\label{isfe}

We first discuss how a partition of a finite graph's vertex set induces a
step-function graphon and how clustering algorithms produce step-function
graphon estimators.
Next we propose iterative step-function estimation (ISFE), an approach to
iteratively improving such estimates by forming a new partition whose classes
contain vertices that have similar edge densities with respect to the old partition.

\subsection{Step-function estimators for graphons}
\label{stepfunction-estimators}

A step-function graphon can be associated with any finite graph given a
partition of its vertices.
Our presentation largely follows \S7.1 and \S9.2 of \citet{MR3012035}, with
modified notation.

A graphon $V$ is called a \defn{step-function} when there is a partition
$\SSS = \{S_1, \ldots, S_k\}$ of $[0,1]$ into finitely many measurable
pieces, called \defn{steps}, such that $V$ is constant on each set $S_i \times S_j$.
Suppose $H$ is a vertex-weighted, edge-weighted graph on $[n]$, with vertex
weights $\alpha_i$
and edge-weights $\beta_{ij}$ for $i,j\in[n]$. Then the
\defn{step-function graphon} $W_H$ associated with $H$ is defined by
$W_H(x, y) = \beta_{ij}$ for $x \in J_i$ and $y \in J_j$, where the steps
$J_1, \ldots, J_n$ form a partition of $[0,1]$ into consecutive intervals of
size
$\frac{\alpha_i}{\sum_{j\in[n]} \alpha_j}$
for $i\in [n]$.
(Consider an unweighted finite graph $G$ to be the weighted graph with vertex
weights $\alpha_i = 1$ and edge weights $\beta_{ij} = G_{ij}$.)

Given a graph $G$ on $[n]$ and vertex sets $X,Y \subseteq [n]$, write
$c_G(X,Y) \defas \sum_{i\in X} \sum_{j\in Y} G_{ij}$ for the number of edges
across the cut $(X, Y)$.
Then the \defn{edge density} in $G$ between $X$ and $Y$ is defined to be
	$e_G(X, Y) \defas  \dfrac{c_G(X, Y)}{|X| \, |Y|};$
when $X$ and $Y$ are disjoint, this quantity is the fraction of possible edges
between $X$ and $Y$ that $G$ contains.

Now suppose $G$ is a graph on $[n]$ and
$\PP = \{P_1 , \ldots , P_k\}$ is a partition of the vertices of $G$ into
$k$ classes.
The \defn{quotient graph} $G/\PP$ is defined to be the weighted graph on $[k]$
with respective vertex weights $|P_i|/n$ and edge weights
$\frac{e_G(P_i, P_j)}{|P_i| \, |P_j|}$.
For our estimation procedure, we will routinely pass from a sampled graph $G$
and a partition $\PP$ of its vertex set to the graphon $W_{G/\PP}$ formed from
the quotient $G/\PP$.

One may similarly define the step-function graphon $V_\SSS$ of $V$ with respect
to a measurable partition $\SSS = \{S_1, \ldots, S_k\}$ of $[0,1]$ as the
step-function graphon of the weighted graph with each vertex weight $\alpha_i$
equal to the measure of $S_i$ and edge weight
$\beta_{ij} = \int_{S_i \times S_j} V(x,y)\, dx\, dy$.

The Frieze--Kannan weak regularity lemma \citep{MR1674741,MR1723039} implies that
every graphon is well-approximated in the cut distance by such step-functions
formed from measurable partitions;
moreover, a bound on the quality of such an approximation is determined by the
number of classes in the
partition, uniformly in the choice of graphon.
For further details, see
Appendix~\ref{appendix-cut}.

\subsection{Graphon estimation via clustering}
\label{clustering}

The partition $\PP$ of a finite graph $G$ described in
Section~\ref{stepfunction-estimators},
which the step-function $W_{G/\PP}$ utilizes, can be formed by
clustering the nodes using some general clustering method, such as
$k$-means \citep{MR0214227},
hierarchical agglomerative clustering \citep{MR0148188}, random assignment, or
simpler clusterings, such as the trivial partition, in which
all vertices are assigned to a single class, or the discrete partition, in
which all vertices are in separate classes.

In Figure~\ref{cluster-fig}, we display the result of estimating a graphon
according to several clustering algorithms.
In all graphon figures, we use a grayscale gradient for values on $[0,1]$,
where darker values are closer to $1$.

Within the graphon estimator literature, several techniques produce
step-functions, but the analysis has generally focused on the choice of
partition size \citep{Olhede01102014} or on the convergence rates for optimal
partitions \citep{2013arXiv1309.5936W, MR3161460, 2014arXiv1410.5837G},
or else the technique requires multiple observations \citep{DBLP:conf/nips/AiroldiCC13}.
Here we aim to exploit structural aspects of graphs, such weak regularity
(i.e., their uniform approximability in the cut distance), via an algorithm for
forming a new partition that improves the step-function estimate $W_{G/\PP}$
produced by any given partition $\PP$.

\begin{figure}[t]
\centering
    \begin{subfigure}[b]{0.18\textwidth}
        \centering
        \includegraphics[scale=.65]{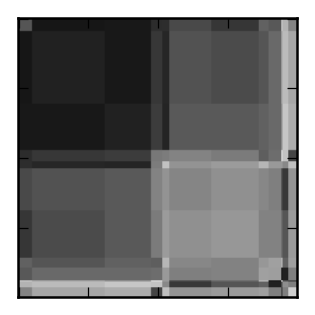}
        \caption{ISFE}
    \end{subfigure}
    \begin{subfigure}[b]{0.18\textwidth}
        \centering
        \includegraphics[scale=.65]{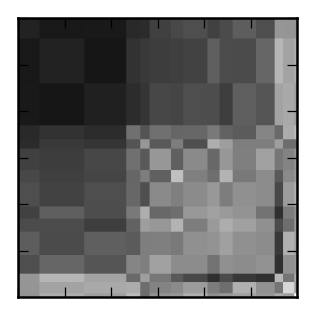}
        \caption{$k$-means}
    \end{subfigure}
    \begin{subfigure}[b]{0.18\textwidth}
        \centering
        \includegraphics[scale=.65]{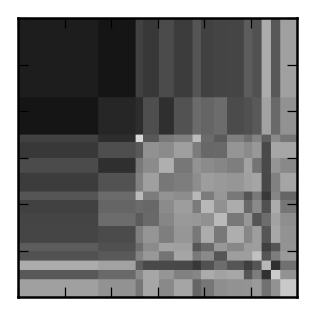}
        \caption{Agglomerative}
    \end{subfigure}
    \begin{subfigure}[b]{0.18\textwidth}
        \centering
        \includegraphics[scale=.65]{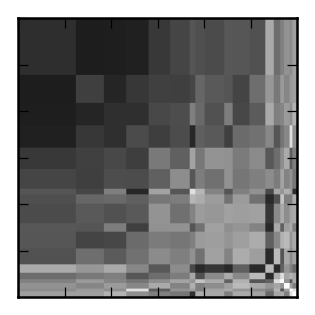}
        \caption{Affinity}
    \end{subfigure}
    \begin{subfigure}[b]{0.18\textwidth}
        \centering
        \includegraphics[scale=.65]{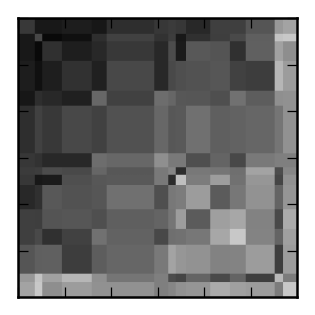}
        \caption{Spectral}
    \end{subfigure}
   \caption{Comparison of the step-function graphons obtained using
	various clustering algorithms on an infinite relational model graphon.
    Columns: (1)
    ISFE was applied to the trivial partition, where all vertices were
    initially assigned to a single bin.
    Clustering was performed
    using the Python package scikit-learn \citep{MR2854348} clustering defaults for
    (2) $k$-means,
    (3) agglomerative clustering,
    (4) affinity propagation,
    and (5) spectral clustering,
    except that the number of clusters was set
    to $k=15$ for each method that uses a fixed cluster size.
    }
    \label{cluster-fig}
\end{figure}

\subsection{Iterative step-function estimation}
\label{isfe-alg-subsec}

In Algorithm~\ref{algs}, we describe
iterative step-function estimation (ISFE), which can be used to produce graphon
function and value estimates.

Given a finite graph $G$, consider the following graphon function estimator procedure:
(a) partition the vertices of $G$ according to some clustering algorithm;
(b) repeatedly improve this partition by iteratively running
Algorithm~\ref{algs} for $T\ge 0$
iterations; and
(c) report the step-function graphon $W_{G/\PP}$, where $\PP$ is the final
partition produced, with its classes sorted according to their average edge
densities.
Let $V$ be a graphon and $n\in\Nats$, and suppose $G$ is a sample of the
$V$-random graph $\G(n, V)$.
The ISFE procedure on $G$ can be evaluated as a graphon function estimate in
terms of MISE by directly comparing $W_{G/\PP}$ to $V$.

ISFE can also be used to produce a graphon value estimate from a graph $G$ on $[n]$.
Let $k$ be the number of classes in an initial partition $\PP$ of $G$.
Implicit in the ISFE procedure is a map $p : [n] \to [k]$ sending each vertex of
$G$ to the index of
its class in $\PP$.
A graphon value estimate is then given by
$\widehat{M} = \bigl(W_{G/\PP}(\frac{2p(i)-1}{2k}, \frac{2p(j)-1}{2k})\bigr)_{i,j\in[n]}$.
In other words, a regular grid $(\frac{2\ell-1}{2k}, \frac{2m-1}{2k})_{\ell, m \in [k]}$
of $k\times k$ points
within $[0,1]^2$ is chosen
as a set of representatives of the piecewise constant regions of
$W_{G/\PP}$,
in some order that corresponds to
how the vertices of $G$ were rearranged into the partition $\PP$.
In a synthetic run, where $G$ is a sample of the $V$-random graph $\G(n, V)$
and we retain the history of how $\PP$ was formed from the values
$M = \bigl(V(U_i, U_j)\bigr)_{i,j\in[n]}$,
MSE can be evaluated by comparing $M$ with $\widehat{M}$.


\begin{Algorithm}[t]
\begin{minipage}{.9\textwidth}
   \begin{algorithmic}
       \State \textbf{Input}: graph $G$,
       initial partition $\PP^{(\text{old})}$, min.\ classes $\ell$,
       decay $d$

           \State \textbf{Output}: new partition $\PP^{(\text{new})}$

                \State Initialize $\QQ=\{Q_1\}, Q_1=\{1\}, c_1 = 1, \epsilon=1$.

                \While{number of classes $|\QQ| < \ell$}

                    \For{vertices $i=2,\ldots, n$}

                    \State Compute weighted edge-densities vector\\
\hspace*{35pt}
$                        e_{i} \defas \bigl\{ \frac{|P_j|}{n}\, e_G(\{i\}, P_j) \bigr\}_j$
$                        \text{for }  j = 1,\ldots,|\PP^{(\text{old})}|.$

                    \State Find $j$ minimizing $L^1$ distance between  
                    vectors $e_{i}$,$e_{c_j}$:\\
\hspace*{35pt}
$j^{*} \defas
\arg\min_{j} d_j(e_{i}, e_{c_j})$, where
                            \\
\hspace*{35pt}
$d_j(e_{i}, e_{c_j})\defas
                            \sum_{r}
                            \bigl|
\frac{|P_r|}{n}
e_G(\{i\}, P_r) -
\frac{|P_r|}{n}
e_G(\{c_j\}, P_r)
\bigr|$.

                        \If{minimum distance $d_{j^{*}}(e_i, e_{c_j}) < \epsilon$}
                        \State
                        Add vertex $i$ to existing class $Q_{j^{*}}$.
                    \Else

                        \State Create new class
                            $Q_{B} = \{i\}$, where $B = |\QQ| + 1$.

                        \State Set centroid $c_{B} = i$.

                        \State Add class $Q_{B}$ to partition $\QQ$.

                    \EndIf
                \EndFor

               \State Reduce $\epsilon$ by the decay parameter:
                $\epsilon \leftarrow \epsilon \cdot d$.
            \EndWhile
            \State Return $\PP^{(\text{new})} \leftarrow \QQ$.
    \end{algorithmic}
\end{minipage}
\label{algorithm}
\caption{\!\!{\bf.}\ \ Single iteration}
\label{algs}
\normalsize
\end{Algorithm}


\begin{figure}[h]
    \centering
    \begin{subfigure}[b]{0.18\linewidth}
        \centering
        \includegraphics[scale=.65]{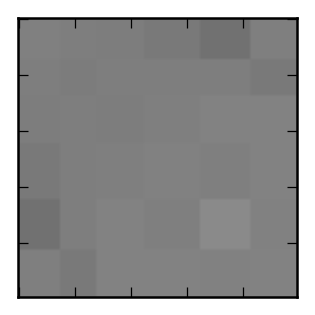}
        \caption{$T=0$}
    \end{subfigure}
    \begin{subfigure}[b]{0.18\linewidth}
        \centering
        \includegraphics[scale=.65]{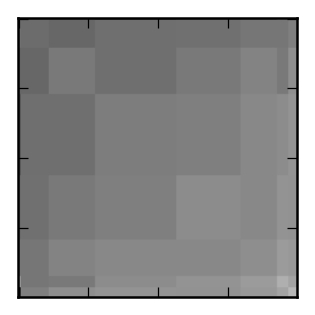}
        \caption{$T=1$}
    \end{subfigure}
    \begin{subfigure}[b]{0.18\linewidth}
        \centering
        \includegraphics[scale=.65]{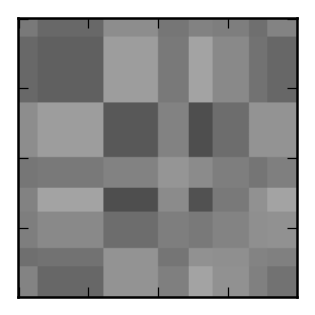}
        \caption{$T=2$}
    \end{subfigure}
    \begin{subfigure}[b]{0.18\linewidth}
        \centering
        \includegraphics[scale=.65]{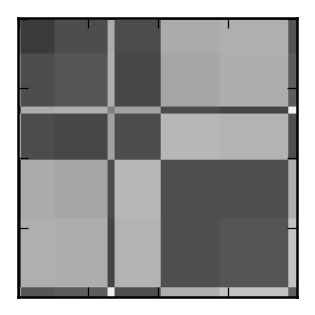}
        \caption{$T=3$}
    \end{subfigure}
    \begin{subfigure}[b]{0.18\linewidth}
        \centering
        \includegraphics[scale=.65]{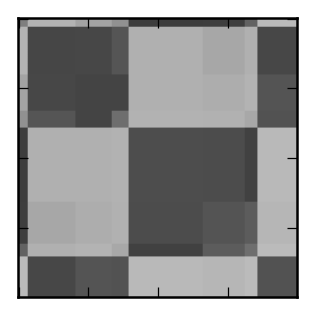}
        \caption{$T=4$}
    \end{subfigure}
    \caption{The first 4 iterations of ISFE on a 200 vertex sample from a SBM
        graphon with $p=0.5, q = (0.7, 0.3)$,
    beginning with a random partition into 6 classes $(T=0)$.}
    \label{iteration-fig}
    \normalsize
\end{figure}

As discussed in Section~\ref{stepfunction-estimators},
by the weak regularity lemma for graphons, every graphon
can be approximated to arbitrary accuracy in cut distance by a step-function,
whose number of steps depends on the desired accuracy and not the graphon.
ISFE seeks to take advantage of this structure.
In an iteration of ISFE, each vertex is grouped with other vertices that are
similar in average edge density with respect to the input partition, weighted by
the size of each class.
In this way, ISFE optimizes average edge densities between classes of the
partition in proportion to the number of vertices they contain.
Each subsequent
iteration seeks to improve upon the previous partition, by using it as the basis
for the next round of density calculations.

Figure~\ref{iteration-fig} demonstrates how ISFE extracts structure from a
graph sampled from an SBM over the course of several iterations, beginning with
a random partition ($T=0$), in which each vertex is independently placed into
one of 6 classes uniformly at random.
(For details on the SBM parameters, see Section~\ref{isfe-examples} below.)
Slight discrepancies in the edge densities between classes in the random
partition of $T=0$ are amplified in iterations $T=1,2$.  The substantial
correlations between the classes of the partition obtained in $T=2$ and
the true block structure allow in $T=3$ to produce a partition each of whose
classes is largely from a single block. This is refined slightly in $T=4$.

\subsection{Examples of ISFE}
\label{isfe-examples}

We now present examples of the behavior of ISFE on certain classes of graphons.
In
Appendix~\ref{appendix-finite-graphs},
we discuss how ISFE performs on step-functions of finite graphs.

\paragraph{Stochastic block model}
\label{sbm-text}

The stochastic block model (SBM)
has been extensively studied from many perspectives;
for a survey of some of the statistical literature as it relates to the graphon
estimation problems, see \citet[\S2.2]{2012arXiv1212.1247C}.

In the stochastic block model, we assume there is some fixed finite number of classes.
We define the SBM graphon with $2$ classes as follows: given parameters $p \in [0,1],
q=(q_0, q_1) \in [0,1]^2$,
we partition $[0,1]$ into two pieces
$P_0, P_1$ of length $p$ and $1-p$, where $p\in[0,1]$.
The value of the graphon is constant on $P_i \times P_i$ with value $q_0$ and constant on $P_i \times P_{1-i}$
with value $q_1$, for $i = 0,1$.

We show the result of ISFE on a graph sampled from an SBM graphon in
Figure~\ref{sbm-fig} for $p=0.5, q = (0.7, 0.3)$.
In the first column, we plot the graphon. We sample a 200
vertex graph from this SBM (column 2) and run ISFE (starting with the trivial partition) on the sample
with $\ell=8$ for $T=5$ iterations (column 3).
We show in the fourth column the original sample rearranged according to the
step-function estimate produced by ISFE.
The last column is the step-function estimate discretized on a $200\times200$
grid and sorted according to increasing $U_i$ value.
(On non-synthetic datasets, where there are no such values $U_i$, it is not possible to
produce this reordering.)
The graphon in this final column is more easily visually compared to the
original graphon in column 1, whereas the original estimate in column 3 shows
more directly the partition found by ISFE.

\begin{figure*}[t]
\centering
\begin{subfigure}[b]{\linewidth}
    \centering
    \includegraphics[scale=.7]{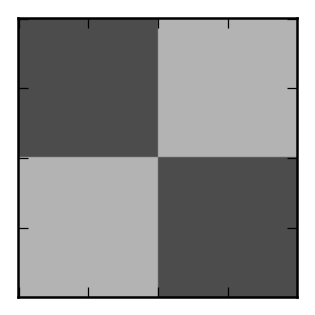}
    \includegraphics[scale=.7]{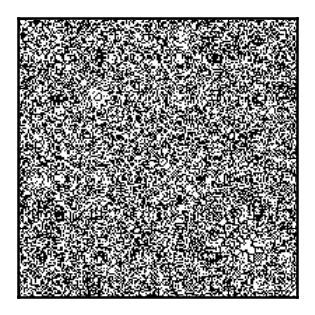}
    \includegraphics[scale=.7]{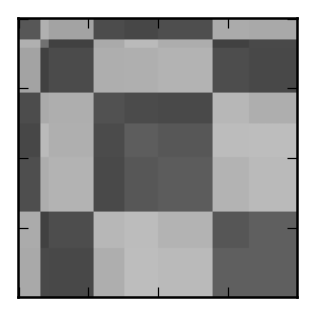}
    \includegraphics[scale=.7]{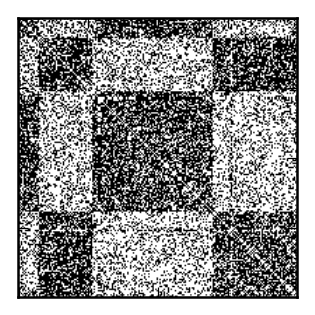}
    \includegraphics[scale=.7]{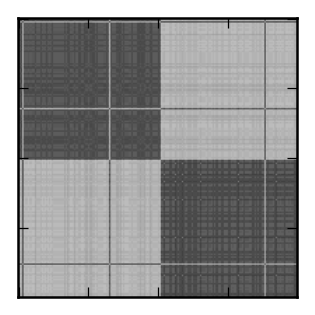}
    \caption{Stochastic block model}
    \label{sbm-fig}
\end{subfigure}
\begin{subfigure}[b]{\linewidth}
    \centering
    \includegraphics[scale=.7]{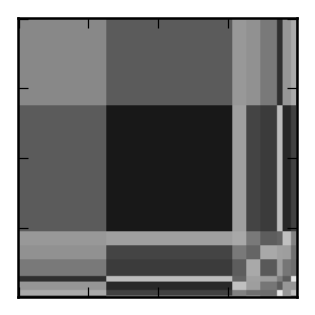}
    \includegraphics[scale=.7]{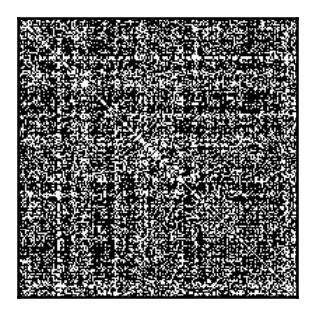}
    \includegraphics[scale=.7]{figs/irm3}
    \includegraphics[scale=.7]{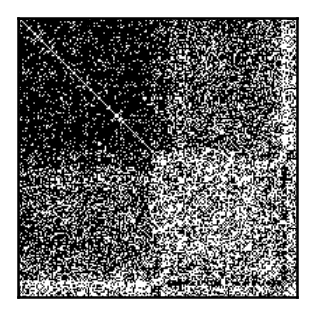}
    \includegraphics[scale=.7]{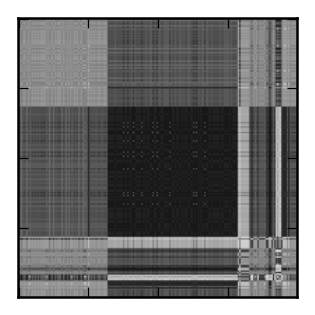}
    \caption{Infinite relational model}
    \label{irm-fig}
\end{subfigure}
\caption{
    Examples of graphon estimation using ISFE.
    Rows:\ (a) an SBM graphon and (b) an IRM graphon.
    Columns:\ (1) the original graphon; (2) a 200 vertex random sample from the graphon;
    (3) ISFE result with $\ell=8$ (SBM), $15$ (IRM), $T=4$ on a
        200-vertex sample; (4) the random sample reordered according to
    the ISFE estimate; (5) ISFE estimate rearranged by increasing $U_i$.
    ISFE was applied to the trivial partition, i.e., all vertices were initially in a
    single class.
}
\label{eg-fig}
\end{figure*}

\paragraph{Infinite relational model}

The infinite relational model (IRM) \citep{C.Kemp:2006:53fd9}
is a non-parametric extension of the SBM,
where the (infinite) partition is generated by a Chinese restaurant process with
concentration parameter $\alpha$.
For a description of the distribution on graphons implicit in this model, see
\citet[Example~IV.1]{DBLP:journals/pami/OrbR14} and
\citet[\S4]{DBLP:conf/nips/LloydOGR12}.
For each class of the partition,
the graphon is constant with value sampled from
a beta distribution with parameters $a, b$.

We show the result of ISFE on a graph sampled from an IRM graphon with
$\alpha=3, a=3, b=2.9$ in
Figure~\ref{irm-fig}. The five columns are analogous to those described above
for Figure~\ref{sbm-fig}.

\section{Analysis of ISFE}
\label{isfe-analysis}

We analyze one aspect of the behavior of a single iteration of a variant of ISFE
(with randomly assigned centroids for ease of analysis, though we expect
greedily chosen centroids, as described in Algorithm~\ref{algs}, to perform even
better) on a stochastic block model.
Consider a $2$-step stochastic block model
graphon with steps of size $p$ and $1-p$ where $p\le \frac12$ and edge densities
$q_0 > q_1$. In this situation we say that a vertex is
\defn{correctly classified} by a class of a partition if it comes from the same
block as the majority of vertices within this class. We can think of the
fraction of vertices
correctly classified by a partition as a proxy for MSE: if a partition's classes
correctly classify a large fraction of vertices, then the
step-function graphon induced by that partition must have small MSE.

Suppose the algorithm begins with a partition of the vertices of a sampled
graph $G$ on $[n]$ into $k$ classes that
correctly classifies some fraction $\tau > 1 - \frac1{4k}$ of the
vertices. We show for arbitrary $\tau'> \tau$, that for sufficiently large
$n$ and $k$, with high probability this iteration correctly classifies a
fraction $\tau'$ of vertices.
While this does not demonstrate how ISFE ``gets started'' from a trivial or
random partition, it does provide insight into how ISFE behaves once a large
fraction of vertices have been correctly classified.

\begin{theorem}
\label{Theorem: Main theorem}
Suppose $\tau > 1- \frac{1}{4k}$ and that the initial partition
of $G$ correctly classifies at least $\tau n$ many vertices.
If $\tau' > \tau$, then for every $\epsilon > 0$ and every $\xi > 0$ such that
\begin{itemize}
\item[(i)]
$q_0-q_1 \geq \frac{3 - (1 - 4k(1-\tau))^{\frac{1}{2}}}{\tau} - 2 + 4 \epsilon
+ \frac{ 1 - (1 - 4k(1-\tau))^{ \frac{1}{2} } }{ n \tau (1 - \tau) }$
\ and
\item[(ii)]
$\epsilon^2\, n  > -12k \, \log\Bigl(\dfrac{1-(\tau' +
\xi)^{\frac{1}{k}}}{2}\Bigr),$
\end{itemize}
the partition obtained by applying (this variant of) ISFE correctly classifies
at least $\tau' n$ many vertices with probability at least
\[
\bigl(1- p^k - (1-p)^k\bigr) \,
 \left(1 - 2 \, \exp\left\{ -\frac{\epsilon^2\,n}{12k}\right\}\right)^{k^2}
 \left(1 - 2 \, \exp\left \{ -\frac{\xi^2\, n}{3} \right \} \right).
\]
\end{theorem}

This shows that if ISFE is
given a partitioning operation which only is guaranteed to correctly classify
a fraction $\tau$ of vertices of an SBM, then iterating ISFE beginning with this
partition ensures that with high probability an arbitrarily large fraction of
vertices are correctly classified, in an appropriate limit of increasingly large
$n$ and $k$.

For the proof, see
Appendix~\ref{appendix-proof}.

\section{Results}
\label{evaluation}

We examine synthetic data sampled from several graphons:
(1) a gradient graphon given by the function
$W(x,y) = \frac{(1-x) + (1-y)}{2}$;
(2) an SBM graphon with $p=0.5$, $q = (0.7, 0.3)$;
(3) an SBM graphon with  $p=0.3$, $q=(0.7,0.3)$; and
(4) an IRM graphon with $\alpha=3, a=3,b=2.9$.
In Figure~\ref{SAS-comparison}, we display the results of ISFE and two other estimators on a 200 vertex
sample from each of the graphons.
The first column displays the graphon, the
second column is the result of ISFE on the sample (starting with the trivial partition), the
third column shows the result of SAS \citep{DBLP:conf/icml/ChanA14}, and the
last column shows the result of USVT \citep{2012arXiv1212.1247C}.
We display in Figure~\ref{figure-plots} the MSE of each estimation method on
samples from the graphons listed above for varying sizes of $n$,
averaged over 50 independent draws.

We evaluate all estimators using the mean squared error (MSE)
given in Equation~\eqref{mse-eq} for the graphon value estimator $\hat{M}$ we
describe below.
For ISFE, we consider the graphon value estimator $\hat{M}$ described in
Section~\ref{isfe-alg-subsec};
the parameter $\ell$ was set to the value that minimized the MSE over
a grid of values;
the gradient graphon used $\ell=20$, the SBM graphon with $p=0.5$
used $\ell = 8$, the SBM graphon with $p=0.3$ used $\ell = 10$, and the IRM graphon
used $\ell = 15$.
For a 200-vertex sample from the IRM graphon, ISFE took 3.4 seconds to
complete for $T=3$ (using 8 GB RAM and 4 cores at 2.8 GHz).
For a fixed partition size, there is a limit to how well one can approximate the
graphon in MSE.
Hence after some number of iterations, ISFE will cease to improve.
We therefore set the number of iterations $T$ to be the first value after which
the MSE no longer improved by at least $10^{-3}$.

We modified SAS by minimizing the total variation
distance using a general-purpose constrained optimization method instead of
using the alternating direction method of multipliers.
For USVT, we follow \citet{DBLP:conf/icml/ChanA14} and first sort the sample by degree.
SAS and USVT
expect graphons with monotonizable degree distribution.
In order to exhibit these estimation techniques for arbitrary graphons, we
take $\hat{M}$ to be a permutation of their estimated matrix, rearrranged so as
to invert the way vertices from the sampled graph were sorted.

For all graphon estimator images, we re-sort the result by increasing $U_i$ so
that the result can be visually compared to the original graphon.
While MSE is directly a measure of similarity between this re-sorted estimate
and the original graphon (evaluated at certain points), in some cases better
upper-bounds on the MISE may be obtained from other rearrangements of the estimate.
Indeed, the smoothed function estimates obtained by SAS and USVT for the gradient
lead to a considerably smaller upper-bound on the MISE than their displayed
re-sorted estimates (and than the original value estimates by USVT).

\begin{figure}[tb!]
\centering
    \begin{subfigure}[b]{\textwidth}
        \centering
        \includegraphics[scale=.7]{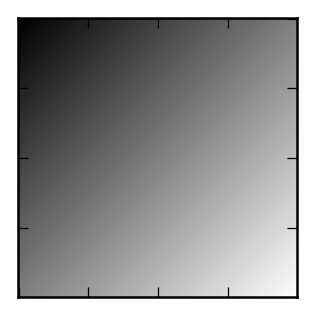}
        \includegraphics[scale=.7]{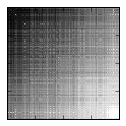}
        \includegraphics[scale=.7]{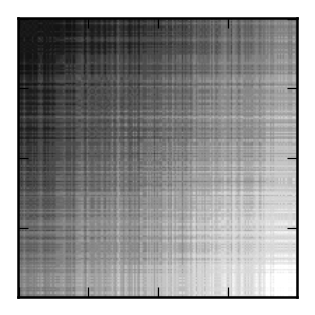}
        \includegraphics[scale=.7]{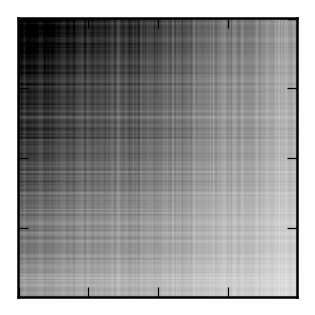}
\vspace*{-3pt}
    \caption{Gradient.
\vspace*{-3pt}
    }
    \label{SAS-comparison-gradient}
    \end{subfigure}
    \\
    \begin{subfigure}[b]{\textwidth}
    \centering
    \includegraphics[scale=.7]{figs/sbm1}
    \includegraphics[scale=.7]{figs/sbm5}
    \includegraphics[scale=.7]{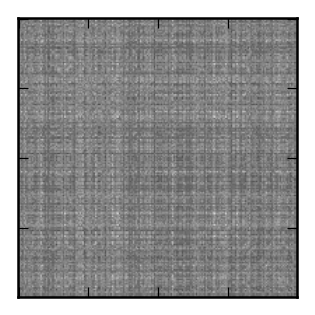}
    \includegraphics[scale=.7]{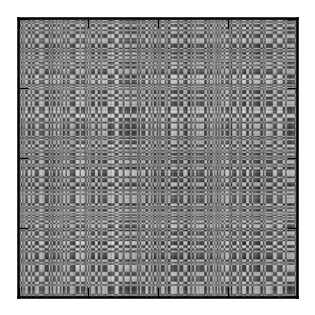}
\vspace*{-3pt}
    \caption{SBM, $p=0.5, q=(0.7,0.3)$. 
\vspace*{-3pt}
    }
    \label{SAS-comparison-SBM-even}
    \end{subfigure}
    \\
    %
    \begin{subfigure}[b]{\textwidth}
    \centering
    \includegraphics[scale=.7]{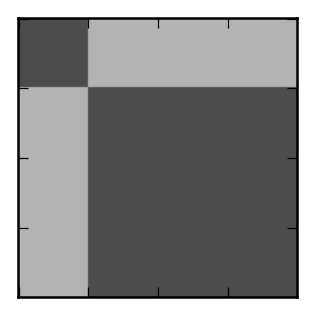}
    \includegraphics[scale=.7]{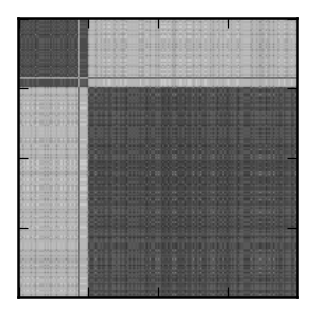}
    \includegraphics[scale=.7]{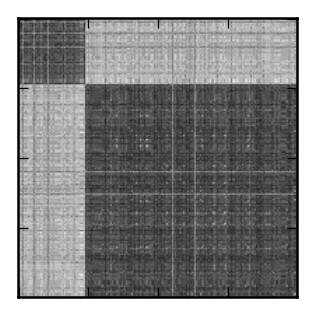}
    \includegraphics[scale=.7]{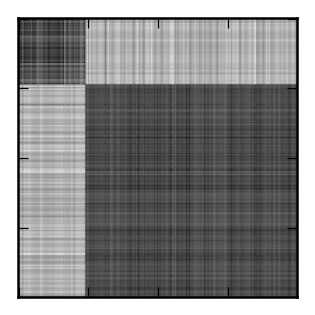}
\vspace*{-3pt}
    \caption{SBM, $p=0.3, q=(0.7,0.3)$. 
\vspace*{-3pt}
    }
    \label{SAS-comparison-SBM-uneven}
    \end{subfigure}
    \\
    %
    \begin{subfigure}[b]{\textwidth}
    \centering
    \includegraphics[scale=.7]{figs/irm1}
    \includegraphics[scale=.7]{figs/irm5}
    \includegraphics[scale=.7]{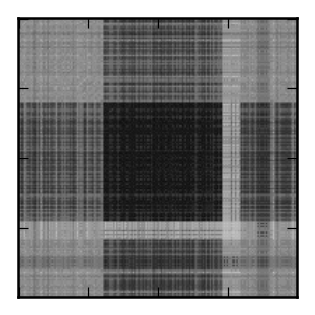}
    \includegraphics[scale=.7]{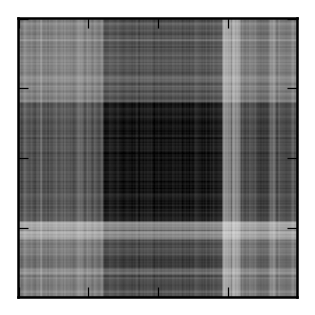}
\vspace*{-3pt}
    \caption{IRM, $\alpha=3, a=3, b=2.9$.
\vspace*{-3pt}
    }
    \label{SAS-comparison-IRM}
    \end{subfigure}
    \\
    \caption{Comparison of graphon estimation methods ISFE, SAS, and USVT.
        Rows:\ (a) Gradient graphon, (b) SBM graphon, $p=0.5, q=(0.7,0.3)$,
        (c) SBM graphon,  $p=0.3, q=(0.7,0.3)$, and (d) IRM graphon, $\alpha=3, a=3, b=2.9$.
        Columns:\ (1) The original graphon, (2) ISFE
        (beginning with the trivial partition), (3) SAS, and (4) USVT.
        All estimator images are sorted by increasing $U_i$.
    }
    \label{SAS-comparison}
\end{figure}

SAS and USVT
perform well
not only for graphons with monotonizable degree distribution, such as gradients
(as in Figure~\ref{SAS-comparison-gradient}), for which SAS was explicitly designed,
but also ones that are monotonizable up to some partition (as in
Figures~\ref{SAS-comparison-SBM-uneven} and \ref{SAS-comparison-IRM}).
However, when the degree
distribution is constant over regions with different structure
(as in Figure~\ref{SAS-comparison-SBM-even}),
SAS and USVT fail to discern this structure.
In contrast, ISFE is able to recover much of the structure after a small number of
iterations, even when it begins with no structural information, i.e., the
trivial partition.

\begin{figure}[t!]
\centering
\begin{subfigure}{0.24\textwidth}
    \includegraphics[scale=0.27]{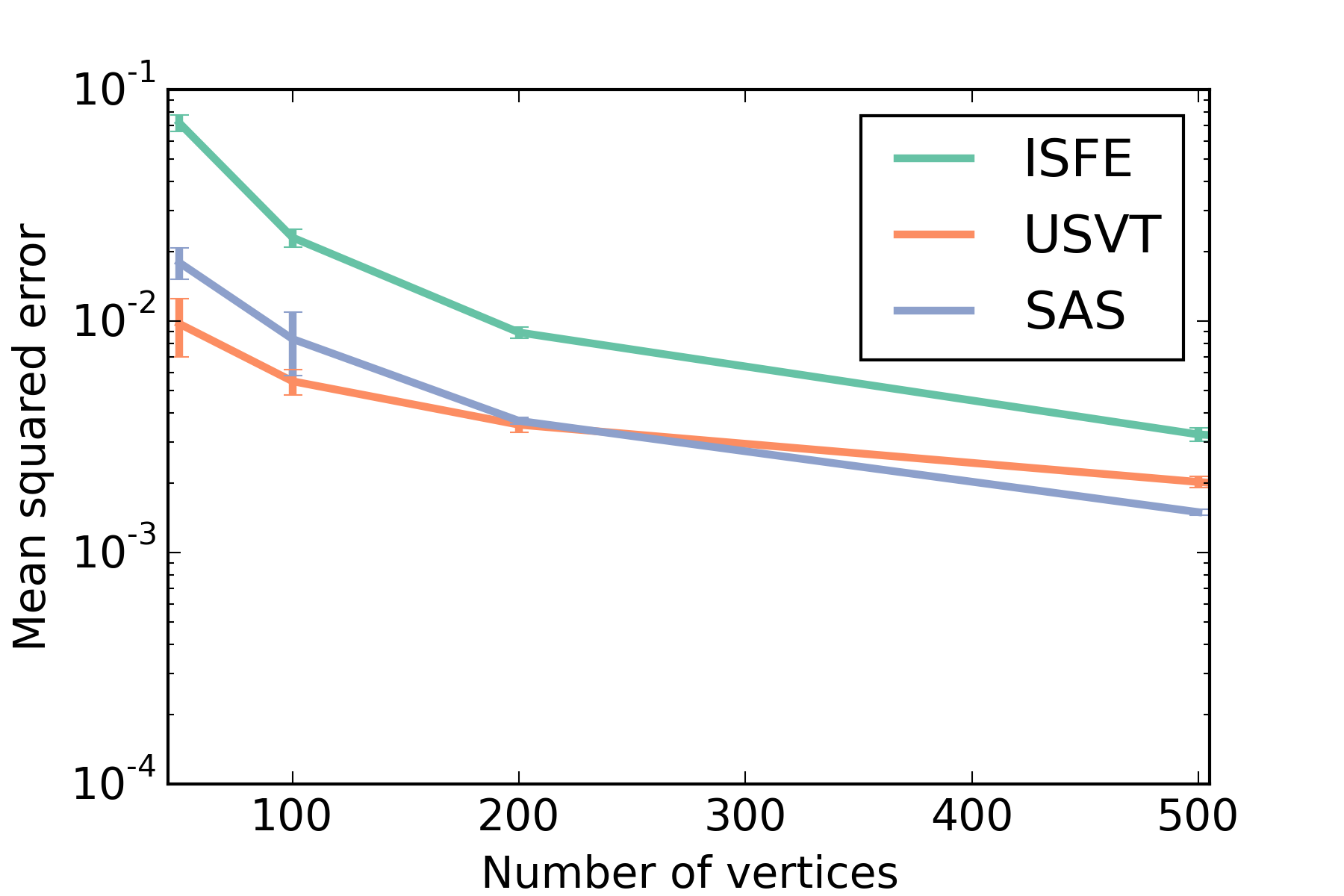}
    \caption{Gradient}
\end{subfigure}
\begin{subfigure}{0.24\textwidth}
    \includegraphics[scale=0.27]{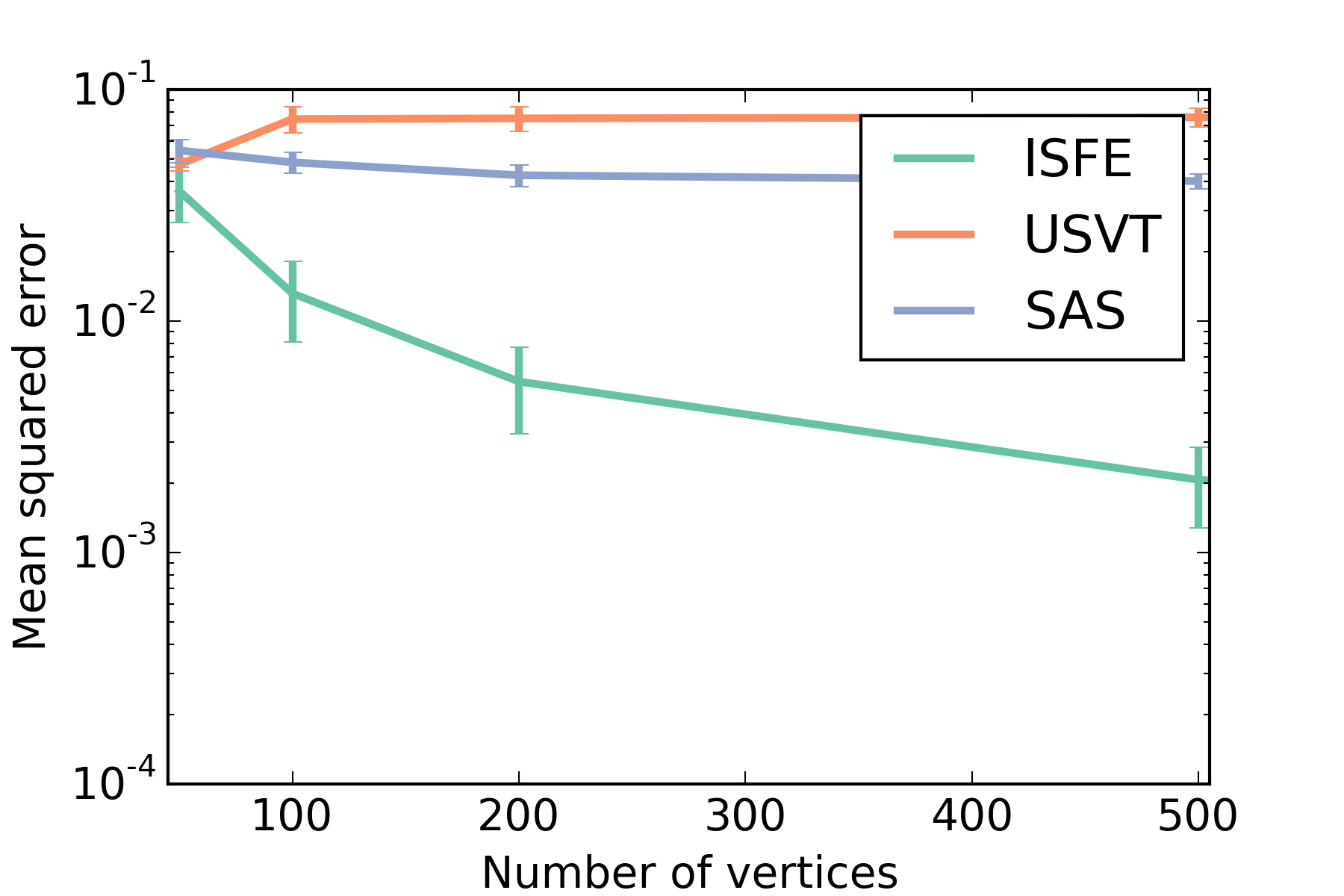}
    \caption{SBM, $p=0.5$}
\end{subfigure}
\begin{subfigure}{0.24\textwidth}
    \includegraphics[scale=0.27]{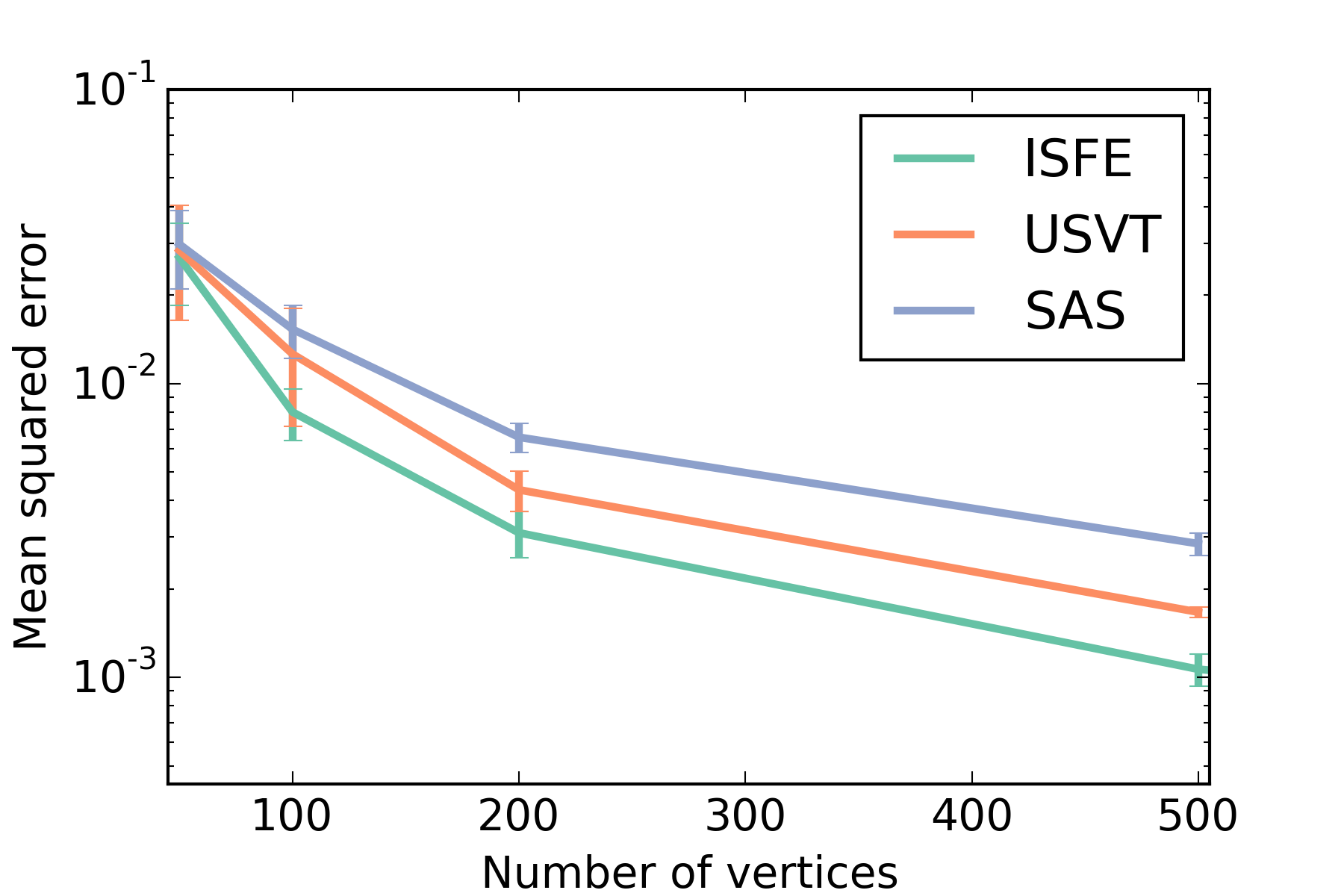}
    \caption{SBM, $p=0.3$}
\end{subfigure}
\begin{subfigure}{0.24\textwidth}
    \includegraphics[scale=0.27]{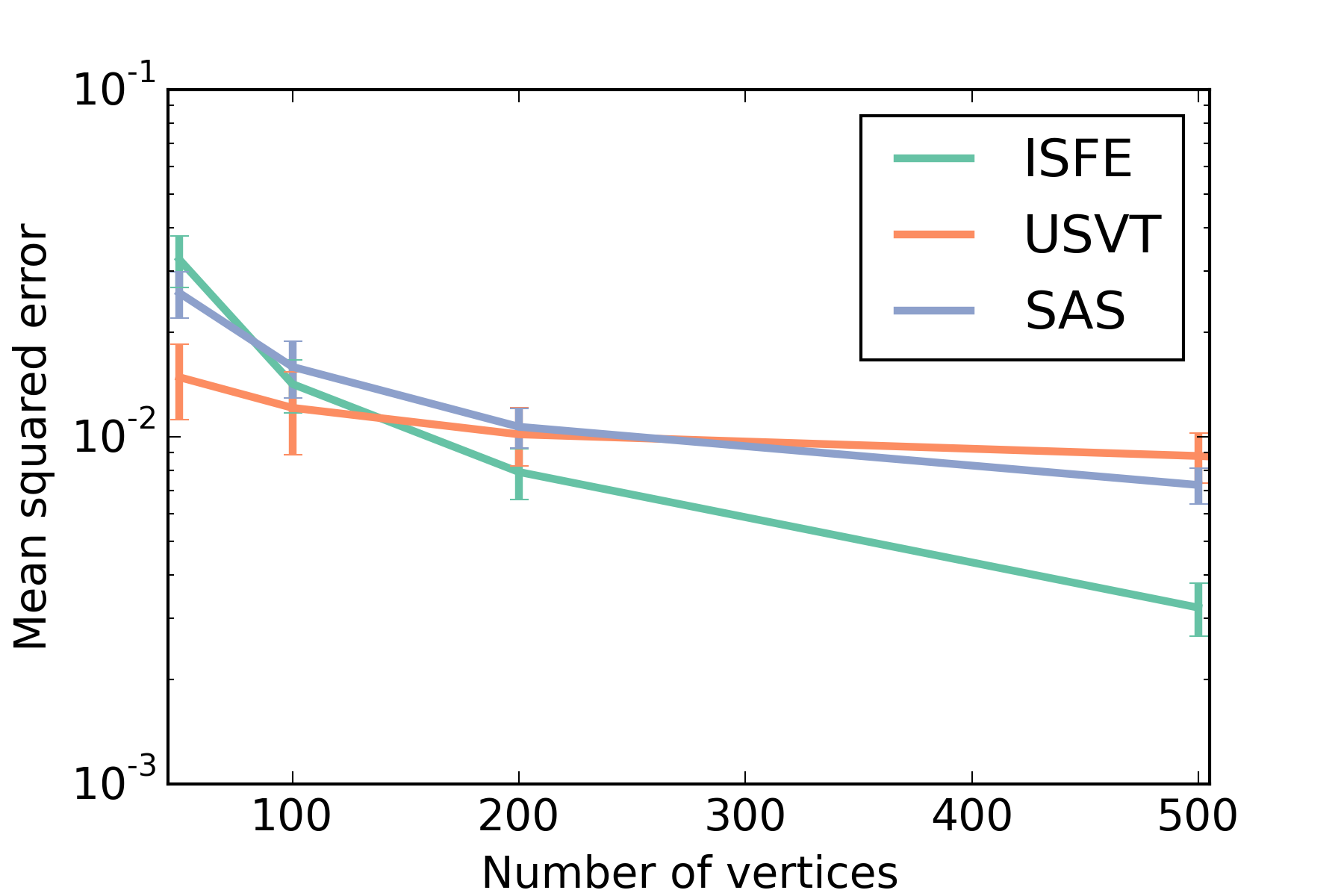}
    \caption{IRM}
\end{subfigure}
\caption{MSE of ISFE, USVT, and SAS estimators on independent samples from
    gradient graphon,
    SBM graphon, $p=0.5, q=(0.7,0.3)$,
    SBM graphon, $p=0.3, q=(0.7,0.3)$,
    IRM graphon, $\alpha=3, a=3, b=2.9$.
    Error bars represent the standard deviation.
    }
\label{figure-plots}
\end{figure}

In
Appendix~\ref{appendix-data}, we examine the result of ISFE on three real-world
social network datasets.
The question of how to model sparse graphs remains an important open problem,
    as we discuss in
    Appendix~\ref{appendix-data}.
To demonstrate our graphon estimator, we sidestep these issues to some extent by
considering a denser subgraph of the original network.

\section{Discussion}
\label{discussion}

While any clustering algorithm naturally induces a graphon estimator, we have
described and shown some of the improvements that may be obtained by grouping
vertices according to their average edge densities with respect to the clusters.
The fact that such improvements are possible is unsurprising for graphs
admitting block
structure (although there are many possible refinements and further analysis of
the algorithm we
describe). Indeed, arbitrary graphons are well-approximated by step-function
graphons in the cut metric, and step-functions graphons (which arise as models
in the
stochastic block model and elsewhere) are well-approximated by such in $L^2$.
A key problem is to devise graphon estimators that further leverage this
structure.

\subsubsection*{Acknowledgments}
The authors would like to thank
Daniel Roy for detailed comments on a draft, and
James Lloyd and
Peter Orbanz
for helpful conversations.

This material is based upon work supported by the United States Air Force and
the Defense Advanced Research Projects Agency (DARPA) under Contract Numbers
FA8750-14-C-0001 and FA8750-14-2-0004.
Work by C.\,F.\ was also made possible through the support of Army Research
Office grant number W911NF-13-1-0212 and a grant from Google.
Any opinions, findings and conclusions or recommendations expressed in this
material are those of the authors and do not necessarily reflect the views of
the United States Air Force, Army, or DARPA.

\newpage

\bibliographystyle{plainnat-mod}

\bibliography{graphons}

\newpage

\begin{appendix}
   \section{Measures of risk for graphons}
\label{appendix-risk}

\subsection{The cut metric}
\label{appendix-cut}

The \emph{cut metric} defines a notion of distance
between two graphs or graphons.
We begin by defining it for finite graphs on the same vertex set,
following \S8.1 and \S8.2 of \citet{MR3012035}.

\begin{definition}
Let $F, G$ be two graphs on $[n]$. The \defn{cut metric} between $F$ and $G$ is given by
\begin{align}
\label{cutmetric}
    d_\Box(F, G) \defas \max_{S,T\subseteq [n]} \frac{|c_F(S,T) - c_G(S,T)|}{n^2}.
\end{align}
\end{definition}
Note that the denominator of Equation~\eqref{cutmetric} is $n^2$ regardless of the size of $S$ and $T$; having large distance between $F$ and $G$ in the cut metric implies that there is some large vertex set on which their respective edge densities differ.

The \defn{cut distance} $\delta_\Box$ between two graphs on different vertex sets of the same size $n$ is then defined to be the minimum of $d_\Box$ over all relabelings of $F$ and $G$ by $[n]$.
While the cut distance can be extended to arbitrary finite weighted graphs on different vertex sets, these definitions are rather technical, and so we instead define the analogous quantity for graphons, from which one may inherit the corresponding notions via step-function graphons.

\begin{definition}
Let $W, V$ be graphons. The \defn{cut metric} between $W$ and $V$ is given by
\[
    d_\Box(W, V) \defas \sup_{S,T\subseteq [0,1]} \left | \int_{S \times T} \bigl( W(x,y) - V(x, y)\bigr)\, dx\, dy \right |,
\]
where $S$ and $T$ range over measurable subsets. The \defn{cut distance} between $W$ and $V$ is defined to be
\[
    \delta_\Box(W, V) \defas \inf_\varphi \, d_\Box(W,V^\varphi),
\]
where $\varphi$ is a measure-preserving transformation of $[0,1]$.
\end{definition}
Note that the cut distance is only a pseudometric, as it is zero for weakly isomorphic graphons.

By the Frieze--Kannan weak regularity lemma \citep{MR1674741,MR1723039},
step-functions approximate graphons in the cut distance, where the required number of steps depends only on the desired accuracy, not the choice of graphon.

\begin{lemma}[Weak regularity {\citep[Lemma~9.3]{MR3012035}}]
\label{FKlemma}
For every $k\geq 1$ and any graph $G$, there is a partition $\PP$ of the vertices of $G$
into $k$ classes such that
\[
        \delta_\Box(G, W_{G/\PP}) \leq \frac{2}{\sqrt{\log k}}.
\]
\end{lemma}
An analogous statement holds for graphons \citep[Corollary~9.13]{MR3012035}.
Step-functions also approximate graphons arbitrarily well in $L^1$ distance
(or $L^2$, as in MSE or MISE), but the convergence is not uniform in the choice of graphon \citep[Proposition~9.8]{MR3012035}.


\subsection{$L^2$ and the cut distance}
\label{appendix-l2-cut}

In graphon estimation by step-functions, we are given a finite graph sampled from a graphon, and need to choose the number of steps by which to approximate the graphon. Under $L^2$ risk, the number of such steps could be arbitrarily large depending on the particular graphon.  But moreover, this number can vary wildly even among graphons that give rise to similar distributions --- which therefore must be close in the cut distance \citep[Lemma 10.31]{MR3012035}.
For example \citep[Example~10.11]{MR3043217}, this can be seen
by considering a constant function graphon $W$ (whose samples are \ER\ graphs) and the step-function graphon $V_k$ induced by a graph of size $k$ sampled from $W$.
(For an illustration of $W$ and $V_k$, see Figure~\ref{ER-sample}.)
Their samples $\G(n, W)$ and $\G(n, V_k)$ have similar distributions, and
indeed $\delta_\Box(W, V_k) = O(1/\sqrt{k})$,
even though the $L^1$ distance between $W$ and $V_k$ is roughly $1/2$ regardless
of $k$ (and hence $L^2$ also does not decrease in $k$).
For this reason, it may be more appropriate to consider risk based on the cut distance, rather than the $L^2$-based MISE, for the function estimation problem for arbitrary graphons.

On the other hand, both the cut metric and $L^1$ can be close for step-functions \citep[Equation (8.15)]{MR3012035}.
Hence even in $L^1$ (or $L^2$), it can be more reasonable to approximate a step-function graphon (as opposed to an arbitrary graphon) by step-function graphons.

Furthermore, while a large $L^2$ distance between graphons does not guarantee that they are far in the cut distance,
a small $L^2$ distance does provide an upper bound on cut distance (and hence on the total variation distance between the
distributions on small subsamples).
Indeed, in many cases (as in our empirical results here), one is in fact able to produce good bounds on $L^2$.

\begin{figure*}[t]
    \centering
    \includegraphics[scale=.9]{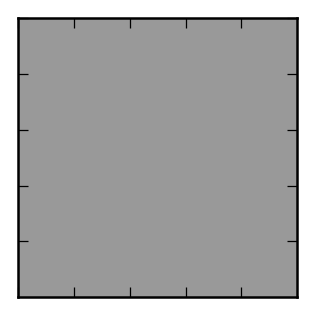}
    \includegraphics[scale=.9]{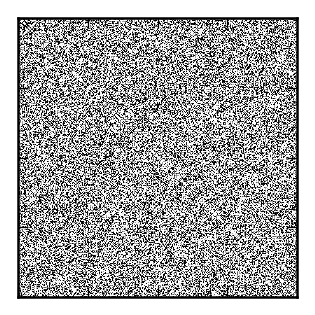}
\caption{Example of graphons that are close in the cut distance but are far in $L^1$ (and $L^2$).
(left) The constant graphon $W = 1/2$. (right) The graphon $V_{1000}$ obtained
as the step-function of a sample from $\G(1000,W)$.
}
\label{ER-sample}
\end{figure*}

\section{Further discussion and analysis of ISFE}

\subsection{ISFE for step-functions of finite graphs}
\label{appendix-finite-graphs}

If one clusters the vertices of a sampled graph discretely, by assigning one class for each vertex, this typically induces a poor graphon estimator, as the reported graphon is merely the step-function of the sample. On the other hand, if we perform even a single iteration of ISFE on such a partition, we obtain an estimator that clusters vertices according to the Hamming distances between their edge vectors. In the case where the original graphon is the step-function of some finite graph on $n$ vertices,
ISFE following such a partition can recover the original
graphon exactly in the sense of MSE, so long as the requested partition size is at least $n$. (Estimation with respect to MISE is limited only by the extent to which the sampled graph distorts the proportion of vertices arising from each vertex of the original.)

We present an example in Figure~\ref{bw-fig}.
In this example, we form the step-function graphon $W_G$ of a graph $G$ with $7$ vertices; its adjacency matrix can be seen via the black and white squares in Figure~\ref{bw-sf}.
ISFE is run on a $70$ vertex sample (Figure~\ref{bw-sampled}) for a single iteration starting from the partition with every vertex in its own class. This single iteration amounts to forming a partition based on Hamming distance on the vector of edges for each vertex (i.e., rows of the adjacency matrix); so long as the requested number of bins is at least 7, the original structure will be recovered.
The resulting graphon from ISFE in Figure~\ref{bw-ham} and the original step-function
graphon in Figure~\ref{bw-sf} are not weakly isomorphic, because some of the 7 steps of the original resulted in slightly more or fewer than 10 vertices in the sample,
but they are very close in the cut distance and, after rearranging by a measure-preserving transformation, in $L^1$. Note that sorting by degree (Figure~\ref{bw-sorted}) garbles the structure among those regions corresponding to distinct vertices of the original having the same degree.

\begin{figure}[t]
    \centering
    \begin{subfigure}[b]{0.2\linewidth}
    \centering
    \includegraphics[scale=.2]{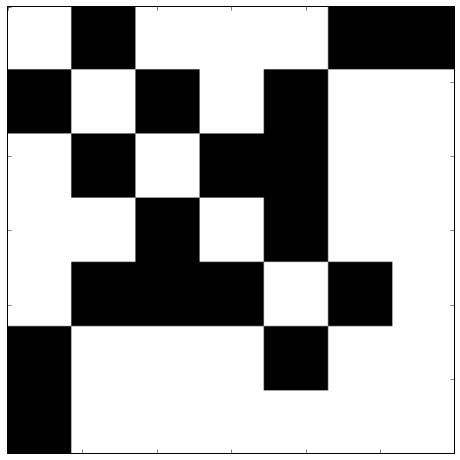}
    \caption{Step-function}
    \label{bw-sf}
    \end{subfigure}
    \begin{subfigure}[b]{0.2\linewidth}
    \centering
    \includegraphics[scale=.2]{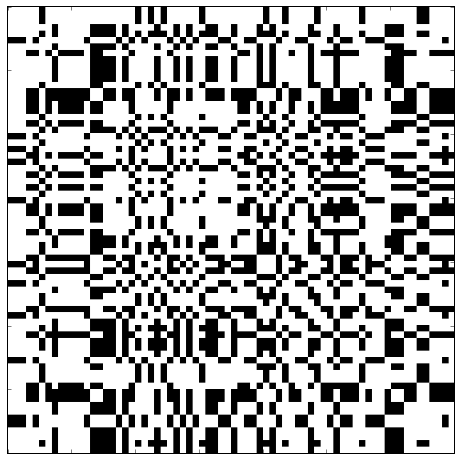}
    \caption{Sampled graph}
	\label{bw-sampled}
    \end{subfigure}
    \begin{subfigure}[b]{0.2\linewidth}
    \centering
    \includegraphics[scale=.2]{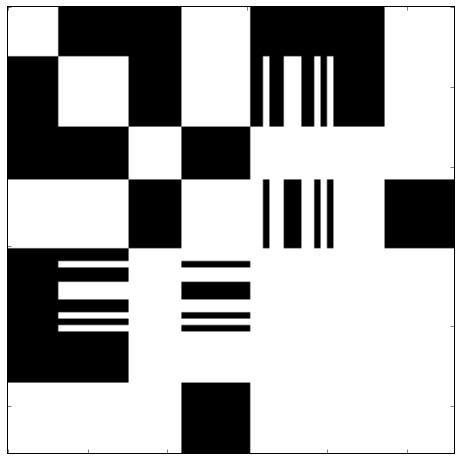}
    \caption{Sorted by degree}
	\label{bw-sorted}
    \end{subfigure}
    \begin{subfigure}[b]{0.2\linewidth}
    \centering
    \includegraphics[scale=.2]{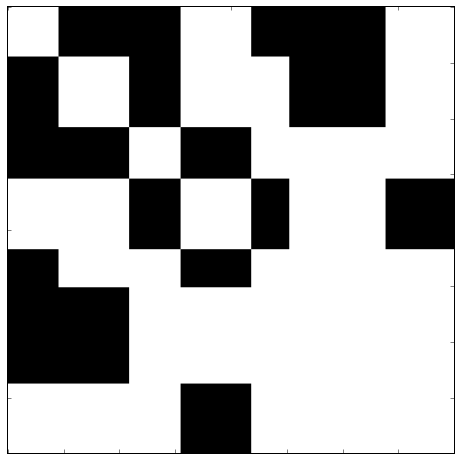}
    \caption{Hamming}
    \label{bw-ham}
    \end{subfigure}
    \caption{
    (a) Black and white step-function graphon of a $7$ vertex graph;
    (b) a 70 vertex sample from the step-function;
    (c) result of sorting the sample by degree;
    (d) result of binning the sample using Hamming distance.
    }
    \label{bw-fig}
\end{figure}

\subsection{ISFE for stochastic block models}
\label{appendix-proof}

We now discuss the behavior of a variant of ISFE on a $2$-step stochastic
block model. We show that given a partition of the vertices, in which a large
portion of the vertices in each class are correctly classified (i.e., the
vertices are in the same class as their true class according to the SBM),
with high probability
the ISFE algorithm will improve (by an arbitrary amount) the fraction of vertices
that are classified correctly.

In this variant of ISFE, centroids are chosen randomly (as described below) rather than greedily; while this makes the algorithm easier to analyze, one expects that ISFE as described and evaluated empirically elsewhere in the paper should perform at least as well in most circumstances.

Let $W$ be a step-function graphon describing a stochastic block model, with $2$ steps $A,B$ of respective sizes $p$ and $1-p$, where $p \leq \frac12$. Let $q_0$ be the edge density within each class and $q_1$ the
edge density across the two classes, and assume $q_0 > q_1$.

Let $G$ be a graph on $[n]$ sampled from $\G(n,W)$. Let $\Gd$ be the (not necessarily simple) graph obtained from $G$ by adding a self-loop at vertex $i$ with probability $q_0$ independently for each $i \in [n]$.
Note that $\Gd$ can be thought of as the result of a modified sampling procedure from $W$, where we instead
assign self-loops with the same probability as other edges within a class.
We will find that many of our calculations are more easily done in terms of $\Gd$.

Considering how each vertex of $G$ was sampled according to $\G(n,W)$,
define the following sets:
\begin{eqnarray*}
[n]_A &\defas& \{x \in [n] \, \st \, x \text{~came from~} A\} \text{~and} \\
\ [n]_B &\defas& \{x \in [n] \, \st \, x \text{~came from~} B\}.
\end{eqnarray*}
Suppose the algorithm begins with a partition $\{C_1,\ldots,C_{k}\}$ of the vertices $[n]$ into $k\ge 2$ classes.
For each $i\in[k]$, define
\[
C^*_i \defas
\begin{cases}
[n]_A  \cap C_i  & \text{if~} \bigl|[n]_A  \cap C_i \bigr| \,\ge\, \bigl|[n]_B  \cap C_i\bigr|,\text{~and} \\
[n]_B  \cap C_i  & \text{otherwise}.
\end{cases}
\]
For $i\in[k]$ we call $C^*_i$ the \defn{majority} of $C_i$,
and for each vertex $x \in C_i$ we say that $x$ is \defn{correctly classified} when $x \in C^*_i$.
Define
\begin{eqnarray*}
K_A &\defas& \{i \in [k] : C^*_i \subseteq [n]_A\}, \text{~and}\\
K_B &\defas& \{i \in [k]: C^*_i \subseteq [n]_B\}.
\end{eqnarray*}
Recall that given a vertex $x \in [n]$, we define its
\defn{weighted edge-density vector} to be
\[
\textstyle
e_x \defas
\bigl
[
\frac{|C_1|}{n}\,
e_G(x, C_1)
, \ldots,
\frac{|C_{k}|}{n}\,
e_G(x, C_{k})
\bigr
]
.
\]

Let $\tau$ be such that
at least a $\tau$-fraction of vertices in $[n]$ are correctly classified by the
partition
$\{C_1,\ldots,C_k\}$, i.e.,
\[
\sum_{i \in [k]}|C^*_i| \geq \tau n.
\]
Given arbitrary $\tau' > \tau$, our goal is to give a lower bound on the
probability that, after applying the variant of ISFE we describe here, the
resulting partition correctly classifies
at least a fraction $\tau'$ of vertices.

We now analyze one iteration of
this variant of ISFE beginning with a partition
$\{C_1,\ldots,C_k\}$.
We create a new partition by first selecting $k$ many centroids
uniformly independently at random from $[n]$ without replacement, and then we
assign every remaining (i.e., non-centroid) vertex $x \in [n]$ to the bin whose
centroid's weighted edge-density vector is closest in $L^1$ to
the weighted edge-density vector $e_x$
(breaking ties uniformly independently at random).

\begin{definition}
For $\delta > 0$, define a class
$C_i$  to be \defn{$\delta$-large} when
$|C_i|\geq \delta \frac{n}{k}$. We say that a $\delta$-large class
$C_i$
is \defn{$\delta$-good} when
$\frac{|C^*_i|}{|C_i|} \geq \delta$  further holds, i.e., a large fraction of its vertices are correctly classified.
Define
\[
D_\delta \defas \{i \in [k] \, \st \, C_i \text{~is~$\delta$-large~and~$\delta$-good}\}.
\]
\end{definition}

Note that for $i\in D_\delta$,
with high probability, the edge density of each vertex $x\in[n]$ with respect to $C_i$ will be close to
its \emph{true} edge density, i.e., $q_0$ if either
\[
x \in [n]_A \text{~and~} C^*_i\subseteq [n]_A
\]
or
\[
x \in [n]_B \text{~and~} C^*_i\subseteq [n]_B,
\]
and $q_1$ otherwise.

\begin{lemma}
\label{Lemma: when always delta-good}
Suppose $C_i$ is $\delta$-large.
If $\delta - \delta^2 \geq k\, (1- \tau)$, then $C_i$ is $\delta$-good.
\end{lemma}
\begin{proof}
By our assumption on $\tau$ we know that
\[
|C_i \setminus C^*_i| \leq n \, (1- \tau).
\]
Hence
\begin{align*}
\frac{|C^*_i|}{|C_i|} &= 1 - \frac{|C_i \setminus C^*_i|}{|C_i|} \ge 1-  \frac{n \, (1- \tau)}{|C_i|} \\
&\ge 1-  \frac{n \, (1- \tau)}{\delta \, \frac{n}{k}} = 1-  \frac{k\, (1- \tau)}{\delta}.
\end{align*}
But $1 - \frac{k\, (1- \tau)}{\delta} \geq \delta$
because $\delta - \delta^2 \geq  k\, (1- \tau)$.
Hence $C_i$ is $\delta$-good.
\end{proof}

In other words, if $k$ is not too large with respect to $\tau$, and $\delta$ is sufficiently large, then every $\delta$-large class must also be $\delta$-good.
Hence if most vertices $x$ are such that
for $\epsilon$ close to $0$ and $\delta$ close to $1$,
for every $\delta$-good
class, the density of $x$ with respect to the majority of that class is within
$\epsilon$ of its expected value, then the
step-function graphon determined by the partition $\{C_1, \ldots, C_k\}$ gives rise to a graphon value estimate of $W$ that
yields a small MSE.
We make this notion precise by defining
$(\epsilon, \delta)$-good vertices.

Throughout the rest of this section, we omit brackets for singleton classes, and write, e.g., $e_G(x, C_i)$ in place of $e_G(\{x\} ,C_i)$.

\vspace*{5pt}
\begin{definition}
For $\epsilon, \delta > 0$ we say that a vertex $x \in [n]_A$ is \defn{$(\epsilon, \delta)$-good} when
for all $i\in D_\delta$, if $i \in K_A$ then
\[
|e_{\Gd}(x, C^*_i) - q_0| < \epsilon
\]
and if $i \in K_B$ then
\[
|e_{\Gd}(x, C^*_i) - q_1| < \epsilon.
\]
Similarly, we say that a vertex $x \in [n]_B$ is
\defn{$(\epsilon, \delta)$-good} when
for all $i\in D_\delta$, if $i\in K_B$ then
\[
|e_{\Gd}(x, C^*_i) - q_0| < \epsilon
\]
and if $i \in K_A$ then
\[
|e_{\Gd}(x, C^*_i) - q_1| < \epsilon.
\]
We let $\GG_{\epsilon, \delta}(x)$ be the event that $x$ is $(\epsilon, \delta)$-good.
\end{definition}

In other words, $x$ is $(\epsilon, \delta)$-good if for every $\delta$-good
class, the density of $x$ with respect to the majority of that class is within
$\epsilon$ of its expected value.

We will begin by showing that if each of the centroids is
$(\epsilon, \delta)$-good (for an appropriate $\epsilon$ and $\delta$),
then each $(\epsilon, \delta)$-good
vertex is correctly classified. This will then reduce
the task of giving
bounds on the probability that at least $\tau' n$ vertices are correctly
classified to that of giving bounds on the probability that at least $\tau' n$ vertices are $(\epsilon, \delta)$-good.

\begin{proposition}
\label{key-expression-prop}
Suppose that each centroid is $(\epsilon, \delta)$-good, and that
at least one centroid is in $[n]_A$ and at least one centroid is in $[n]_B$.
Further suppose that
\[
\textstyle
q_0-q_1 \geq 2 (\frac{2-\delta}\tau -1 + 2\epsilon + \frac{k}{n\delta\tau} ).
\]
Then the weighted edge-density vector of each $(\epsilon, \delta)$-good vertex
of $[n]_A$ is closer in $L^1$ to that of some centroid in $[n]_A$ than to that
of any centroid in $[n]_B$.
Similarly, the weighted edge-density vector of each $(\epsilon,\delta)$-good
vertex of $[n]_B$ is closest to that of a centroid in $[n]_B$.
\end{proposition}
\begin{proof}
For each $(\epsilon, \delta)$-good vertex $d_A \in [n]_A$, and each $(\epsilon,\delta)$-good vertex $d_B \in [n]_B$, each $i_A \in K_A\cap D_\delta$, and $i_B \in K_B \cap D_\delta$ we have
\begin{align*}
    \tau(q_0 - \epsilon) &\leq e_\Gd(d_A, C_{i_A}) \leq \tau(q_0 + \epsilon) + 1-\tau, \\
    \tau(q_1 - \epsilon) &\leq e_\Gd(d_A, C_{i_B}) \leq \tau(q_1 + \epsilon) + 1-\tau, \\
    \tau(q_1 - \epsilon) &\leq e_\Gd(d_B, C_{i_A}) \leq \tau(q_1 + \epsilon) + 1-\tau, \text{~and}\\
    \tau(q_0 - \epsilon) &\leq e_\Gd(d_B, C_{i_B}) \leq \tau(q_0 + \epsilon) + 1-\tau.
\end{align*}
Hence we also have
\begin{align*}
\textstyle
    \tau(q_0 - \epsilon) - \frac{1}{|C_{i_A}|} &\leq e_G(d_A, C_{i_A}) \leq \tau(q_0 + \epsilon) + 1-\tau, \\
\textstyle
    \tau(q_1 - \epsilon) - \frac{1}{|C_{i_B}|} &\leq e_G(d_A, C_{i_B}) \leq \tau(q_1 + \epsilon) + 1-\tau, \\
\textstyle
    \tau(q_1 - \epsilon) - \frac{1}{|C_{i_A}|} &\leq e_G(d_B, C_{i_A}) \leq \tau(q_1 + \epsilon) + 1-\tau, \text{~and}\\
\textstyle
    \tau(q_0 - \epsilon) - \frac{1}{|C_{i_B}|} &\leq e_G(d_B, C_{i_B}) \leq \tau(q_0 + \epsilon) + 1-\tau,
\end{align*}
where the differences come from the fact that in $\Gd$ we may add self-loops, and that $d_A$ and $d_B$ are themselves in some class. Further, as $|C_i| \geq \frac{n}{k} \delta$, we have
\begin{align*}
\textstyle
    \tau(q_0 - \epsilon) - \frac{k}{n\delta} &\leq e_G(d_A, C_{i_A}) \leq \tau(q_0 + \epsilon) + 1-\tau, \\
\textstyle
    \tau(q_1 - \epsilon) - \frac{k}{n\delta}  &\leq e_G(d_A, C_{i_B}) \leq \tau(q_1 + \epsilon) + 1-\tau, \\
\textstyle
    \tau(q_1 - \epsilon) - \frac{k}{n\delta}  &\leq e_G(d_B, C_{i_A}) \leq \tau(q_1 + \epsilon) + 1-\tau, \text{~and}\\
\textstyle
    \tau(q_0 - \epsilon) - \frac{k}{n\delta}  &\leq e_G(d_B, C_{i_B}) \leq \tau(q_0 + \epsilon) + 1-\tau.
\end{align*}

We first consider the $L^1$ distance between the weighted edge-density vectors
of a vertex from $[n]_A$ and a vertex $x$.
Again assume the vertex $d_A \in [n]_A$ is $(\epsilon, \delta)$-good,
and suppose $x\in [n]_A$ is an $(\epsilon, \delta)$-good vertex.
For $i_A \in K_A$ we have
\begin{eqnarray*}
|e_G(x, C_{i_A}) - e_G(d_A, C_{i_A})|
&\leq& (\tau (q_0 + \epsilon) + 1-\tau) - \tau(q_0-\epsilon) +
    \tfrac{k}{n\delta}\\
    &=& 2\tau\epsilon + 1-\tau + \tfrac{k}{n\delta},
\end{eqnarray*}
and for $i_B \in K_B$ we have
\begin{eqnarray*}
|e_G(x, C_{i_B}) - e_G(d_A, C_{i_B})|
&\leq& (\tau (q_1 + \epsilon) + 1-\tau) - \tau(q_1-\epsilon) +
    \tfrac{k}{n\delta}\\
    &=&  2\tau\epsilon + 1-\tau + \tfrac{k}{n\delta}.
\end{eqnarray*}
Further,
\[
\textstyle
(k-1)\, (1-\delta)\, \frac{n}{k} < (1-\delta)\, n,
\]
and so at most $(1-\delta) \, n$ many vertices are not in classes with
respect to which they are $\delta$-good.
Hence the $L^1$ distance between the weighted edge-density vectors
\[
\textstyle
e_x =
\bigl
[
\frac{|C_1|}{n}\,
e_G(x, C_1)
, \ldots,
\frac{|C_k|}{n}\,
e_G(x, C_k)
\bigr
]
\]
and
\[
\textstyle
e_{d_A} =
\bigl
[
\frac{|C_1|}{n}\,
e_G(d_A, C_1), \ldots,
\frac{|C_k|}{n}\,
e_G(d_A, C_k)
\bigr
]
\]
is equal to
\begin{align*}
    \frac1n\sum_{i\in[k]} |C_i|\, \bigl|e_G(x, C_i) - e_G(d_A,C_i) \bigr|,
\end{align*}
which is at most
\[
\textstyle
\frac1n \bigl(
(1-\delta)\, n
+
\sum_{i\in[k]} |C_i|\,
(2\tau\epsilon + 1-\tau + \frac{k}{n\delta})
\bigr)
.
\]
Hence the $L^1$ distance is at most
\[
(1- \delta) +
2\tau\epsilon + 1-\tau + \frac{k}{n\delta},
\]
as $\sum_{i\in[k]} |C_i| = n$.

Now let $d_B \in [n]_B$ be $(\epsilon,\delta)$-good.
We now consider the $L^1$ distance between the weighted edge-density vectors
$e_{d_B}$ and $e_x$.
As $q_0 > q_1$, for $i_A \in K_A$ we have
\begin{eqnarray*}
|e_G(x, C_{i_A}) - e_G(d_B, C_{i_A})|
&\geq& \tau(q_0 - \epsilon) - \tfrac{k}{n\delta} - (\tau(q_1+\epsilon) + 1-\tau))\\
    &=& \tau(q_0-q_1 - 2\epsilon) - \tfrac{k}{n\delta} - 1+\tau,
\end{eqnarray*}
and for $i_B \in K_B$ we have
\begin{eqnarray*}
|e_G(x, C_{i_B}) - e_G(d_B, C_{i_B})|
&\geq& \tau(q_0 - \epsilon) - \tfrac{k}{n\delta} - (\tau(q_1+\epsilon) + 1-\tau))\\
    &=& \tau(q_0-q_1 - 2\epsilon) - \tfrac{k}{n\delta} - 1+\tau.
\end{eqnarray*}
Therefore the $L^1$-distance between the weighted edge-density vectors $e_x$ and
\[
\textstyle
e_{d_B} =
\bigl
[
\frac{|C_1|}{n}\,
e_G(d_B, C_1), \ldots,
\frac{|C_k|}{n}\,
e_G(d_B, C_k)
\bigr
]
\]
is at least
\[
\textstyle
\frac1n \bigl(
- (1- \delta) n
+
\sum_{i\in[k]} |C_i|\,
(\tau(q_0-q_1 - 2\epsilon) - \frac{k}{n\delta} - 1+\tau)
\bigr)
\ = \
- (1-\delta) +
\tau(q_0-q_1 - 2\epsilon) - \frac{k}{n\delta} - 1+\tau.
\]

In particular, if
\begin{eqnarray*}
\textstyle
- (1-\delta) +
\tau(q_0-q_1 - 2\epsilon) - \frac{k}{n\delta} - 1+\tau
\geq
(1-\delta) +
2\tau\epsilon + 1-\tau + \frac{k}{n\delta},
\end{eqnarray*}
then $e_x$ is closer in $L^1$ to $e_{d_A}$ than to $e_{d_B}$
whenever $d_A \in [n]_A$ and $d_B\in [n]_B$ are $(\epsilon, \delta)$-good.
But this inequality is equivalent to our hypothesis,
\begin{eqnarray*}
\textstyle
q_0-q_1
\geq
2 (\frac{2-\delta}\tau -1 +
2\epsilon + \frac{k}{n\delta\tau} ).
\end{eqnarray*}
Hence, as we have assumed that all centroids are $(\epsilon, \delta)$-good,
the vertex $x$ is placed in a class whose centroid is in $[n]_A$.
A similar argument shows that if $x \in [n]_B$ is $(\epsilon, \delta)$-good
then $x$ is placed in a class whose centroid is in $[n]_B$.
\end{proof}

Recall that $\GG_{\epsilon, \delta}(x)$ is the event that a vertex $x$ is
$(\epsilon,\delta)$-good.
Notice that for each $x_A, y_A \in [n]_A$,
\[
\Pr(\GG_{\epsilon, \delta}(x_A)) = \Pr(\GG_{\epsilon, \delta}(y_A)),
\]
and for each $x_B, y_B \in [n]_B$,
\[\Pr(\GG_{\epsilon, \delta}(x_B)) = \Pr(\GG_{\epsilon, \delta}(y_B)).
\]
We now want to give a lower bound for these values.
For $n \in \Nats$ and  $\zeta \in [0,1]$,
let $X_{n, \zeta} \sim \frac1n \Binomial(n, \zeta)$, and for $\epsilon > 0$ define
\[
\Bin(n, \zeta, \epsilon) \defas \Pr\bigl(\bigl|X_{n, \zeta} - \zeta\bigr| \,<\, \epsilon\bigr).
\]
Note that we also have
\[
\Bin(n, \zeta, \epsilon) = \Pr\bigl(\bigl|n \, X_{n, \zeta} - n \, \zeta\bigr| \,<\, n\, \epsilon\bigr),
\]
i.e., the probability that the number of successful trials differs from the expected number by at most $n\, \epsilon$.

It is then easily checked, using Chernoff's bound, that the following inequality holds:
\[
\Bin(n, \zeta, \epsilon) \geq 1 - 2 \, \exp\left\{ -\frac{\epsilon^2\, n}{3\zeta} \right\}.
\]
We now use this inequality to bound the probability that a given vertex is
$(\epsilon,\delta)$-good.

\begin{lemma}
\label{Lemma: Lower bound prob good}
For all $x \in [n]$,
\[
\Pr\bigl(\GG_{\epsilon, \delta}(x) \bigr) \ge \left(1 - 2 \, \exp\left\{ -\frac{\epsilon^2\, \delta^2 n}{3k} \right\}\right)^k.
\]
\end{lemma}
\begin{proof}
Note that when $i \in K_A$ is such that $C_i$ is $\delta$-good, we have $\frac{|C^*_i|}{|C_i|} \geq \delta$ and $|C_i|\geq \delta \frac{n}{k}$, and so $|C^*_i| \geq \delta^2 \frac{n}{k}$.

Let $\fourEvent{J}{H}{x}{i}$ denote the event that
$x \in [n]_J$ and $i \in K_H$, where $J, H \in \{A, B\}$.
Observe that conditioning on $\fourEvent{A}{A}{x}{i}$,
we have that $e_\Gd(x, C^*_i)$ has the same distribution as
$X_{|C^*_i|, q_0}$, and so
$\Expect\bigl(e_\Gd(x,C^*_i) \, \big |\, \fourEvent{A}{A}{x}{i}\bigr) = q_0$ and
\[
\Expect\bigl(c_\Gd(x,C^*_i)\, \big |\, \fourEvent{A}{A}{x}{i} \bigr)
= \Expect\bigl( |C^*_i| \, e_\Gd(x,C^*_i) \, \big |\, \fourEvent{A}{A}{x}{i} \bigl)
= q_0\,|C^*_i|.
\]
We therefore also have
\begin{eqnarray*}
\Pr\bigl(| e_\Gd(x,C^*_i) - q_0 | < \epsilon \ \big |\ \fourEvent{A}{A}{x}{i} \bigr) &=& \Bin(|C^*_i|, q_0, \epsilon) \\
	&\geq& 1 - 2 \, \exp\left\{ -\frac{\epsilon^2\, |C^*_i|}{3 q_0} \right\} \\
         &\geq& 1 - 2 \, \exp\left\{ -\frac{\epsilon^2\, |C^*_i|}{3} \right\} \\
         &\geq& 1 - 2 \, \exp\left\{ -\frac{\epsilon^2\, \delta^2 n}{3k} \right\}.
\end{eqnarray*}

A similar argument shows that we have
\[
\Pr\bigl(| e_\Gd(x,C^*_i) - q_1 | < \epsilon \ \big |\ \fourEvent{A}{B}{x}{i} \bigr)  \geq
1 - 2 \, \exp\left\{ -\frac{\epsilon^2\, \delta^2 n}{3k} \right\},
\]
\[
\Pr\bigl(| e_\Gd(x,C^*_i) - q_1 | < \epsilon \ \big |\ \fourEvent{B}{A}{x}{i} \bigr)  \geq
1 - 2 \, \exp\left\{ -\frac{\epsilon^2\, \delta^2 n}{3k} \right\},
\]
and
\[
\Pr\bigl(| e_\Gd(x,C^*_i) - q_0 | < \epsilon \ \big |\ \fourEvent{B}{B}{x}{i} \bigr)  \geq
1 - 2 \, \exp\left\{ -\frac{\epsilon^2\, \delta^2 n}{3k} \right\}.
\]

For a given $x$ and function $Z \colon [k] \to \{q_0,q_1\}$, the events
\[
    \bigl\{| e_\Gd(x,C^*_i) - Z(i) | < \epsilon \   \st \  i \in [k] \bigr\}
\]
are independent.
Hence,
since $|D_\delta|\leq k$,
for any $x \in [n]$
 we have the lower bound
\[
\left(1 - 2 \, \exp\left\{ -\frac{\epsilon^2\, \delta^2 n}{3k} \right\}\right)^k
\leq
\Pr(\GG_{\epsilon, \delta}(x))
\]
on the probability that $x$ is $(\epsilon, \delta)$-good.
\end{proof}

Proposition~\ref{key-expression-prop}
reduces the problem of bounding the probability that a large number of
vertices are correctly classified to that of bounding the probability that
(for appropriate $\epsilon, \delta$) all centroids are $(\epsilon, \delta)$-good
and that a large fraction of vertices are $(\epsilon, \delta)$-good.

Lemma~\ref{Lemma: Lower bound prob good}
bounds the probability that any single vertex is $(\epsilon, \delta)$-good.
If $x_1, \dots, x_r$ were such that the events
$\GG_{\epsilon, \delta}(x_1), \dots, \GG_{\epsilon, \delta}(x_r)$ were
independent, then this would yield a bound on the probability that
$\bigwedge_{ i \in [r]}\GG_{\epsilon, \delta}(x_i)$ holds.

In general, though, the events
$\GG_{\epsilon, \delta}(x_1), \dots, \GG_{\epsilon, \delta}(x_r)$ are
not independent --- and indeed they can interact in a complicated way.
However, conditioning on
$\GG_{\epsilon, \delta}(x_1), \dots, \GG_{\epsilon, \delta}(x_r)$
can only increase the probability that a given $\GG_{\epsilon, \delta}(y)$ holds,
as we now make precise.

\begin{lemma}
Suppose $x_1, \dots, x_r, y \in [n]$. Then
\[
\textstyle
\Pr(\GG_{\epsilon, \delta}(y)) \leq
\Pr\bigl( \GG_{\epsilon, \delta}(y) \, \big | \, \bigwedge_{i \in [r]}
\GG_{\epsilon, \delta}(x_i)\bigr).
\]
\end{lemma}
\begin{proof}
For each $i \in [r]$, let $h_i$ be such that $x_i \in C_{h_i}$,
let $H_i \defas \{j \,\st\, h_{j} = h_i\}$,
let $\mathfrak{H}_i$ be the event that
\[
\bigl|e_{\Gd}(y, C^*_{h_i}) -\Expect(e_{\Gd}(y, C^*_{h_i}))\bigr| < \epsilon
\]
holds, and
let $\mathfrak{J}_i$ be the event that
$\bigwedge_{\ell \in H_i} \GG_{\epsilon, \delta}(x_\ell)$
holds.

Observe that $\{ \mathfrak{H}_k \,\st\, k \in [r] \}$ is a set of independent
events.
Also observe that for all $i \in [r]$,
the event $\mathfrak{H}_i$ is independent of
$\{ \mathfrak{J}_\ell \, \st \, \ell \in [r] \text{ and } H_\ell\neq H_i \}$.

Hence as
\begin{align*}
\textstyle
\Pr(\mathfrak{H}_i) \leq \Pr(\mathfrak{H}_i \given \mathfrak{J}_i)
\end{align*}
holds for all $i \in [r]$, we have
\begin{align*}
\textstyle
\Pr(\bigwedge_{i \in D_\delta} \mathfrak{H}_i)
\leq \Pr(\bigwedge_{i \in D_\delta} \mathfrak{H}_i \given
\bigwedge_{i \in [r]} \mathfrak{J}_i)
.
\end{align*}
But $\bigwedge_{i \in D_\delta} \mathfrak{H}_i = \GG_{\epsilon,\delta}(y)$
and
$\bigwedge_{i \in [r]} \mathfrak{J}_i = \bigwedge_{i\in[r]}
\mathfrak{G}_{\epsilon, \delta}(x_i)$,
and so we are done.
\end{proof}

As a consequence, we have the following immediate corollary.

\begin{corollary}
\label{Corollary: Prob good is at least prob good independent}
Suppose $x_1, \dots, x_r \in [n]$. Then
\[
\textstyle
\Pr\bigl(\bigwedge_{i \in [r]}\GG_{\epsilon, \delta}(x_i)\bigr) \geq
\prod_{i \in [r]} \Pr(\GG_{\epsilon, \delta}(x_i)),
\]
i.e., the probability that $r$ many elements are $(\epsilon,\delta)$-good
is at least the probability of $r$ many independent Bernoulli samples.
\end{corollary}

We now use these results along with Proposition~\ref{key-expression-prop}
to bound the probability of correctly classifying
every $(\epsilon, \delta)$-good vertex.

\begin{lemma}
\label{Lemma: Probability (epsilon, delta) good implies correctly classified}
Suppose that
\[
\textstyle
q_0-q_1 \geq 2 (\frac{2-\delta}\tau -1 + 2\epsilon + \frac{k}{n\delta\tau} ),
\]
and let $p_k \defas 1- p^k - (1-p)^k$.
Then with probability at least
\[
p_k\, \left(1 - 2 \, \exp\left\{ -\frac{\epsilon^2\, \delta^2 n}{3k}
\right\}\right)^{k^2},
\]
every $(\epsilon, \delta)$-good vertex is correctly classified.
\end{lemma}

\begin{proof}
By Corollary~\ref{Corollary: Prob good is at least prob good independent}
and Lemma~\ref{Lemma: Lower bound prob good}, the probability that
all $k$ of the centroids are $(\epsilon, \delta)$-good is at least
\[
\prod_{i \in [k]} \Pr(\GG_{\epsilon, \delta}(x_i)) \geq
\left(1 - 2 \, \exp\left\{ -\frac{\epsilon^2\, \delta^2 n}{3k} \right\}\right)^{k^2}.
\]

Further, $p_k$ is the probability that at least one centroid is in $[n]_A$ and
least one centroid is in $[n]_B$.
Therefore with probability at least
\[
p_k \, \left(1 - 2 \, \exp\left\{ -\frac{\epsilon^2\, \delta^2
n}{3k}\right\}\right)^{k^2},
\]
the conditions of Proposition~\ref{key-expression-prop} hold, and
every $(\epsilon, \delta)$-good vertex is correctly classified.
\end{proof}

Using
Corollary~\ref{Corollary: Prob good is at least prob good independent} and
Lemma~\ref{Lemma: Lower bound prob good},
we may also show that, with high probability, a large number of vertices are
$(\epsilon,\delta)$-good.

\begin{corollary}
\label{Corollary: Bound most elements are (epsilon, delta) good}
For $\xi > 0$
we have the following bound:
\[\Pr\bigl(|\{y: \GG_{\epsilon, \delta}(y)\}| \geq n (\alpha - \xi)\bigr) \ \geq
\ 1 - 2 \, \exp\left\{ -\frac{\xi^2\, n}{3} \right\},\]
where
$\alpha = \left(1 - 2 \, \exp\left\{ -\frac{\epsilon^2\, \delta^2 n}{3k} \right\}\right)^k$.
\end{corollary}
\begin{proof}
First observe that
$\Pr(\GG_{\epsilon, \delta}(x)) \geq \alpha$ for all $x \in [n]$,
by Lemma~\ref{Lemma: Lower bound prob good},
and recall that $\Bin(n, \alpha, \xi)$ is the probability that the average of
$n$ Bernoulli random variables with weight $\alpha$ is within $\xi$ of its expected value.
Hence
\[
\Pr\bigl(|\{y: \GG_{\epsilon, \delta}(y)\}| \geq n (\alpha - \xi)\bigr) \ \geq \ \Bin(n, \alpha, \xi),
\]
by
Corollary~\ref{Corollary: Prob good is at least prob good independent}.
As before, we have
\[
\Bin(n, \alpha, \xi) \geq  1 - 2 \, \exp\left\{ -\frac{\xi^2\, n}{3\alpha}
\right\}.
\]
Hence
\[
\Bin(n, \alpha, \xi) > 1 - 2 \, \exp\left\{ -\frac{\xi^2\, n}{3}\right\},
\]
and so the result follows.
\end{proof}

In particular, with probability at least $1 - 2 \, \exp\left\{ -\frac{\xi^2\, n}{3\alpha} \right\}$, at least an $\alpha-\xi$ fraction of vertices are correctly classified.

Finally, we put all of these calculations together to obtain the following theorem.

\begin{theorem}
\label{Theorem: Main theorem}
Suppose $\tau > 1- \frac{1}{4k}$  and $\{C_1, \dots, C_k\}$ is a partition
of the vertices of $G$ that correctly classifies at least $\tau n$ many vertices.
If $\tau' > \tau$, then for every $\epsilon > 0$ and every $\xi > 0$ such that
\begin{equation*}
\label{thm-eq1}
\tag{$*$}
q_0-q_1 \geq \frac{3 - (1 - 4k(1-\tau))^{\frac{1}{2}}}{\tau} - 2 + 4 \epsilon
+ \frac{ 1 - (1 - 4k(1-\tau))^{ \frac{1}{2} } }{ n \tau (1 - \tau) }
\end{equation*}
and
\[
\label{thm-eq2}
\tag{$\dagger$}
\epsilon^2\, n  > -12k \, \log\Bigl(\dfrac{1-(\tau' + \xi)^{\frac{1}{k}}}{2}\Bigr),
\]
the partition obtained by applying (this variant of) ISFE correctly classifies at least $\tau' n$ many vertices with probability at least
\[
p_k
 \left(1 - 2 \, \exp\left\{ -\frac{\epsilon^2\,n}{12k}\right\}\right)^{k^2} \left(1 - 2 \, \exp\left \{ -\frac{\xi^2\, n}{3} \right \} \right),
\]
where $p_k = 1- p^k - (1-p)^k$.
\end{theorem}
\begin{proof}
Let $\delta = \frac{1 + (1 - 4k(1-\tau))^{\frac{1}{2}}}{2}$ and notice that
$\delta - \delta^2 =  k(1- \tau)$.
We can then apply
Lemma~\ref{Lemma: when always delta-good} to conclude that
whenever $|C_i| \geq \delta \frac{n}{k}$ holds, $C_i$ is $\delta$-good.
The inequality \eqref{thm-eq1} is equivalent to
\[
q_0-q_1 \geq 2 \frac{2-\delta}\tau - 2 + 4\epsilon + 2\frac{k}{n\delta\tau}.
\]
Therefore the hypothesis of
Lemma~\ref{Lemma: Probability (epsilon, delta) good implies correctly classified}
is satisfied, and so with probability at least
\[
p_k
 \left(1 - 2 \, \exp\left \{ -\frac{\epsilon^2\, \delta^2 n}{3k}
 \right\}\right)^{k^2},
\]
every $(\epsilon, \delta)$-good vertex is correctly classified.
Because $\tau > 1 - \frac{1}{4k}$, we have $\delta > \frac{1}{2}$.
Hence the probability of correct classification is at least
\begin{align*}
p_k \, \left(1 - 2 \, \exp \left \{ -\frac{\epsilon^2\,n}{12k} \right
\}\right)^{k^2}.
\end{align*}

By the condition
\eqref{thm-eq2}
and because $\delta > \frac{1}{2}$,
we have
\[
\epsilon^2 \delta^2\, n > \frac{1}{4}\epsilon^2 n > -3k \, \log\Bigl(\frac{1-{(\tau' + \xi)}^{\frac{1}{k}}}{2}\Bigr).
\]
Rearranging this inequality, we obtain
\[
\alpha = \left(1 - 2 \, \exp\left\{ -\frac{\epsilon^2\, \delta^2 n}{3k} \right\}\right)^k > \tau' +\xi.
\]
Applying Corollary \ref{Corollary: Bound most elements are (epsilon, delta) good} we see that with probability at least $1 - 2 \, \exp\left\{ -\frac{\xi^2\, n}{3} \right\}$ at least $n (\tau' + \xi)$ vertices are $(\epsilon, \delta)$-good.

Hence we know that at least $n \tau'$ many vertices are correctly
classified after applying our variant of ISFE, with probability at least
\[
p_k
 \left(1 - 2 \,\exp\left\{ -\frac{\epsilon^2\,n}{12k} \right\}\right)^{k^2}  \left(1 - 2 \, \exp\left\{ -\frac{\xi^2\, n}{3} \right\} \right),
\]
as desired.
\end{proof}

For this variant of ISFE,
Theorem~\ref{Theorem: Main theorem}
describes certain aspects of its
long-term behavior, i.e., as the size of the graph $n$ tends to infinity.

In particular, suppose we fix $k$ and consider those values of $\tau'$
to which Theorem \ref{Theorem: Main theorem} applies, i.e.,
the improvement in the fraction of correctly classified vertices that can be
obtained after applying ISFE.
Note that as $n$ approaches infinity, not only can we find values for $\epsilon$
and $\xi$ such that $\tau'$ becomes arbitrarily close to 1,
but also the probability of ISFE correctly classifying at least
a $\tau'$ fraction of vertices is bounded by $p_k = 1-p^k - (1-p)^k$, which is
the probability that at least one centroid is in $[n]_A$ and at least one
centroid is in $[n]_B$. Finally, letting $k$ vary again, note that
the probability $p_k \rightarrow 1$ as $k \rightarrow \infty$.

Thus, we show in Theorem~\ref{Theorem: Main theorem} that in the limit, where
both the size $n$ of the graph and number of classes $k$ in the partition
approach infinity (in an appropriate relationship),
with high probability this variant of ISFE
will correctly classify an arbitrarily large fraction of the vertices of a $2$-step SBM (if started with an initial partition that correctly classifies enough vertices).

\section{Real-world datasets}
\label{appendix-data}

\begin{figure*}[!htb]
\centering
\begin{center}
\begin{subfigure}[b]{\linewidth}
    \centering
    \includegraphics[scale=.25]{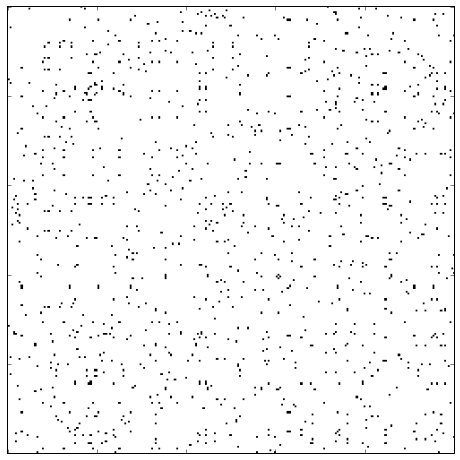}
    \includegraphics[scale=.25]{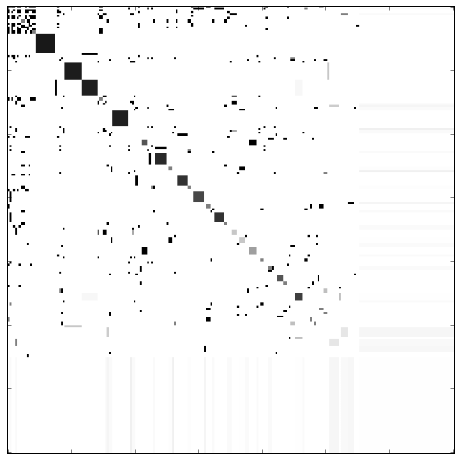}
    \includegraphics[scale=.25]{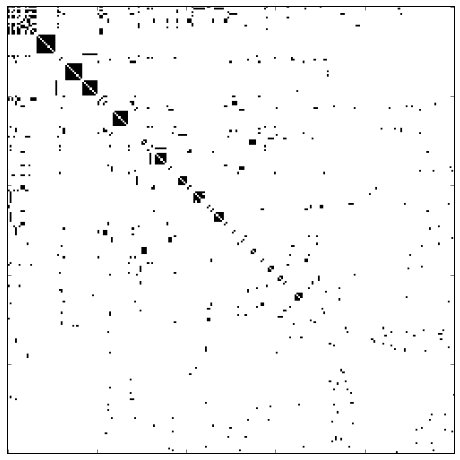}
    \caption{NIPS co-authorship}
\end{subfigure}
\begin{subfigure}[b]{\linewidth}
    \centering
    \includegraphics[scale=.25]{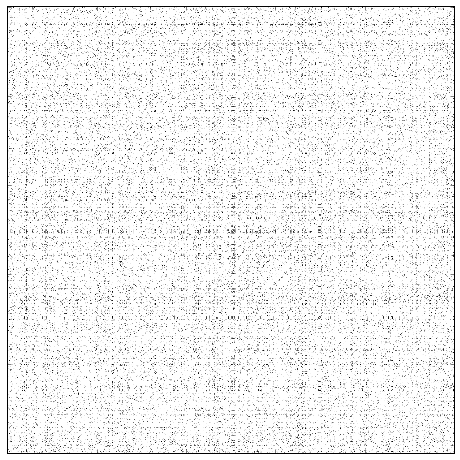}
    \includegraphics[scale=.25]{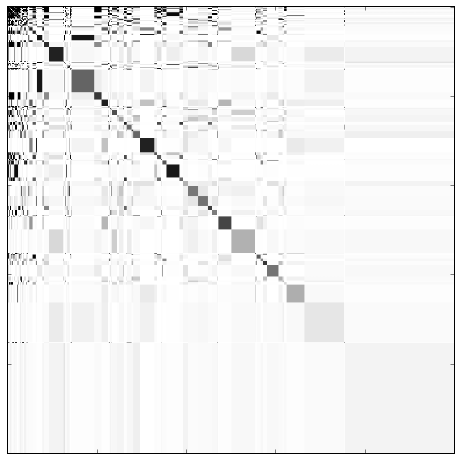}
    \includegraphics[scale=.25]{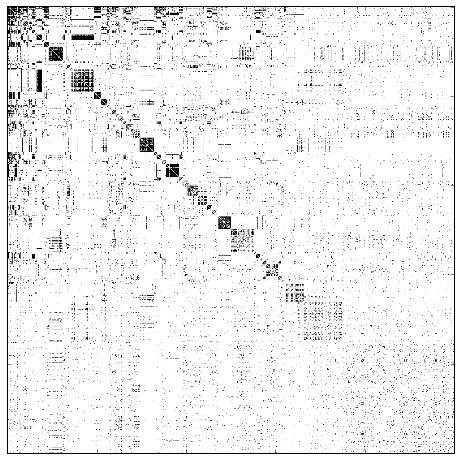}
\caption{ca-AstroPh}
\end{subfigure}
\begin{subfigure}[b]{\linewidth}
    \centering
     \includegraphics[scale=.25]{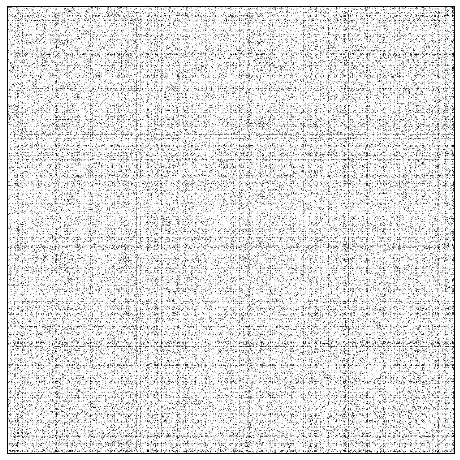}
     \includegraphics[scale=.25]{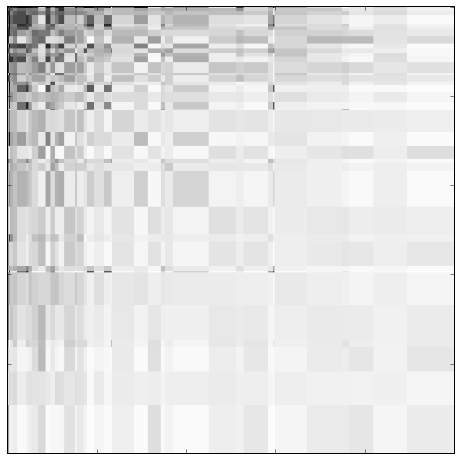}
     \includegraphics[scale=.25]{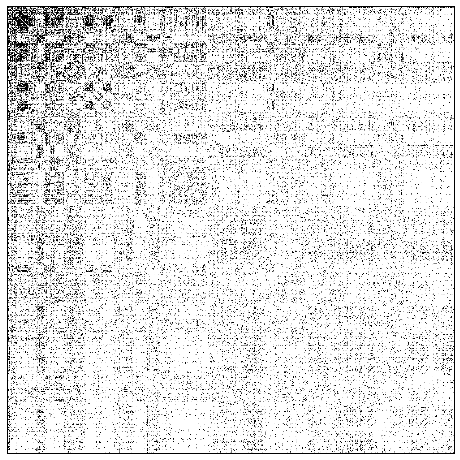}
    \caption{soc-Epinions1}
\end{subfigure}
\end{center}
\caption{ISFE results on real-world datasets for NIPS co-authorship, Astrophysics arXiv co-authorship, and epinions trust network. Columns: (1) A denser subset of
    the original graph; (2) estimated ISFE graphon;
    (3) adjacency matrix rearranged according to ISFE estimate.
}
\label{real-fig}
\end{figure*}

We examine three real-world social network datasets, considering a denser subgraph
constructed by taking the top $K$ highest-degree vertices and the edges between them,
for reasons we describe below.
We randomized the order of the vertices for each graph before running the ISFE
algorithm, which we present in Figure~\ref{real-fig}.

Many real-world networks, such as those arising from co-authorship, social
interactions, etc., are not well-modeled as exchangeable graphs,  as they tend
to exhibit power-law degree distributions, ``small-world'' phenomena such as
short path lengths, and other properties that generally hold only for sparse
sequences of graphs (having $o(n^2)$ edges among $n$ vertices, which is not
possible for non-empty exchangeable graphs). For a detailed discussion, see
\citet[\S VII]{DBLP:journals/pami/OrbR14}.

One approach to modeling sparse graphs using graphons is the
Bollob\'as--Janson--Riordan model
\citep{MR2337396}, where edges are independently
deleted from an exchangeable graph to achieve the desired edge density.
Although this process does not exhibit many of the above real-world phenomena
\citep[Example~VII.4]{DBLP:journals/pami/OrbR14}, the behavior of graphon
estimators on graphs sampled in this way has been considered
\citep{MR2906868, 2013arXiv1309.5936W}.

Here we avoid these issues to some degree by considering a denser subset of the original graph.

\begin{enumerate}
\item NIPS co-authorship dataset \citep{chechik2007eec}:
This dataset is an undirected network of co-authorships in the NIPS conference
from Proceedings 1--12, with 2,037 vertices and 1,740 edges.
We choose $K=234$ for the denser subset, which has been studied in other work
\citep{DBLP:conf/nips/MillerGJ09, DBLP:conf/icml/PallaKG12}.
For the ISFE parameters, we set $T=8, \ell=95$, initializing it with a $90$
cluster $k$-means partition.

\item ca-AstroPh co-authorship dataset \citep{Newman16012001}:
This dataset is an undirected network of co-authorships between scientists posting
pre-prints on the Astrophysics E-Print Archive between Jan 1, 1995 and December
31, 1999 with 18,772 vertices and 396,160 edges.
We choose $K=1000$, and
set the ISFE parameters to $T=8$, $\ell=160$, initializing it with a $150$ cluster
partition from $k$-means.

\item Epinions dataset \citep{pedro}:
This dataset is a who-trusts-whom network of Epinions.com with
75,879 vertices, 508,837 edges.
We work with the undirected graph obtained by symmetrizing the original undirected graph,
choose $K=1000$, and set $T=8$, $\ell=40$, initializing it with a $35$ cluster
partition from $k$-means.
\end{enumerate}

\end{appendix}

\end{document}